\documentclass[11pt]{amsart}

\usepackage{epsfig,amsmath,amsfonts,latexsym,mathrsfs,amssymb,flafter}
\usepackage{color}
\usepackage{overpic}
 \usepackage{palatino}
 \usepackage{hyperref}
 \usepackage{longtable}
 \usepackage[nocompress]{cite}

\newtheorem{theorem}{Theorem}[section]
\newtheorem{proposition}[theorem]{Proposition}
\newtheorem{lemma}[theorem]{Lemma}
\newtheorem{corollary}[theorem]{Corollary}
\newtheorem{conjecture}[theorem]{Conjecture}
\newtheorem{question}[theorem]{Question}

\theoremstyle{definition}

\newtheorem{remark}[theorem]{Remark}

\theoremstyle{remark}

\newcommand{\s}{\mathfrak{s}}

\newcommand{\TT}{\mathfrak{t}}

\newcommand{\tlambda}{{\tilde\lambda}}

\newcommand{\spinc}{{\mbox{spin$^c$} }}

\newcommand{\zee}{\mathbb{Z}}
\newcommand{\arr}{\mathbb{R}}
\newcommand{\enn}{\mathbb{N}}
\newcommand{\cue}{\mathbb{Q}}
\newcommand{\cee}{\mathbb{C}}
\newcommand{\F}{\mathbb{F}}
\renewcommand{\H}{\mathcal{H}}
\newcommand{\eff}{\mathcal{F}}

\newcommand{\K}{{\mathcal K}}

\newcommand{\T}{{\mathcal T}}

\renewcommand{\S}{{\mathcal S}}

\renewcommand{\L}{{\mathcal L}}

\newcommand{\dbar}{{\overline{\partial}}}
\newcommand{\ts}{\textstyle}
\newcommand{\tk}{{\otimes k}}
\newcommand{\tF}{{\widetilde F}}

\newcommand{\hfhat}{\widehat{HF}}

\newcommand{\hfkhat}{\widehat{HFK}}

\newcommand{\cptwobar}{{\overline{\cee P}^2}}

\newcommand{\tC}{{\widetilde{C}}}
\newcommand{\tomega}{{\tilde\omega}}
\DeclareMathOperator{\tb}{tb}
\DeclareMathOperator{\rot}{rot}

\DeclareMathOperator{\Int}{Int}

\DeclareMathOperator{\Skel}{Skel}
\DeclareMathOperator{\st}{St}
\DeclareMathOperator{\symp}{symp}

\newcommand{\hmbar}{{\overline{HM}}}
\newcommand{\hmfrom}{\widehat{HM}}
\DeclareFontFamily{U}{mathx}{\hyphenchar\font45}
\DeclareFontShape{U}{mathx}{m}{n}{
      <5> <6> <7> <8> <9> <10>
      <10.95> <12> <14.4> <17.28> <20.74> <24.88>
      mathx10
      }{}
\DeclareSymbolFont{mathx}{U}{mathx}{m}{n}
\DeclareFontSubstitution{U}{mathx}{m}{n}
\DeclareMathAccent{\widecheck}{0}{mathx}{"71}
\newcommand{\hmto}{\widecheck{HM}}

\theoremstyle{plain}

\begin{document}

\title{On {W}einstein domains in symplectic manifolds}

\author{Thomas E. Mark}

\author{B\"{u}lent Tosun}

\address{Department of Mathematics \\ University of Virginia \\ Charlottesville \\ VA}

\email{tmark@virginia.edu}

\address{Department of Mathematics\\ University of Alabama\\Tuscaloosa\\AL}

\email{btosun@ua.edu}

\begin{abstract} We prove that a Weinstein domain symplectically embedded in a closed symplectic manifold always admits symplectic hypersurfaces in its complement, possibly after a deformation. As a consequence, we obtain an obstruction for a closed 3-dimensional manifold to arise as the boundary of a Weinstein domain in a class of symplectic 4-manifolds that includes many symplectic rational surfaces. A particular application is that no Brieskorn homology sphere bounds a Weinstein domain symplectically embedded in a rational surface diffeomorphic to $S^2\times S^2$ or to $\cee P^2\# k \cptwobar$, for any $k\leq 7$, despite the fact that many Brieskorn spheres bound Stein domains holomorphically embedded in these rational surfaces. Several families of Brieskorn spheres are obtained that do not bound a Weinstein domain in any 4-manifold with a ``positive'' symplectic structure. Such Weinstein domains do exist in certain positive symplectic rational surfaces when $k\geq 8$, though their topology is significantly constrained. 
\end{abstract}

\maketitle

\section{Introduction}

\subsection{Background and statement of some results}\label{introsec1}

If $(X,\omega)$ is a symplectic manifold, a {\it contact type hypersurface} in $X$ is a smooth codimension 1 submanifold $Y$ such that there exists a vector field $V$ in a neighborhood of $Y$ that is transverse to $Y$, and is also Liouville in the sense that $\L_V\omega = \omega$. In this case $V$ induces a contact structure $\xi = \ker(\iota_V\omega|_Y)$ on $Y$ that is independent of the choice of $V$. If $Y$ is separating, then the closure of the component of $X - Y$ for which $V$ is outward-directed is a strong symplectic filling of $(Y, \xi)$ that we call the ``convex side'' of $Y$. 

The contact type condition was introduced by Weinstein \cite{weinstein79}, who observed that under certain circumstances the characteristic foliation of a contact type hypersurface admits a closed orbit. These existence results for periodic orbits were greatly generalized and extended in subsequent years: for example it was shown by Viterbo \cite{viterbo87} that for any contact type hypersurface in standard symplectic $\arr^{2m}$, the Reeb vector field---which generates the characteristic foliation---has at least one closed orbit; and in dimension 3, celebrated work of Taubes \cite{taubesweinconj} shows that a Reeb vector field always has a closed orbit. In another direction, since contact type hypersurfaces always have a standard symplectic neighborhood, they are particularly useful in constructions of symplectic manifolds that involve cutting and gluing along hypersurfaces. This paper is motivated by the following question:

\begin{question}\label{CTquestion} If $(X,\omega)$ is a symplectic $n$-manifold, which smooth closed $(n-1)$-manifolds $Y$ admit an embedding in $X$ as a separating contact type hypersurface?
\end{question} 	

A fundamental case is that of $X = \arr^{2m}$ with its standard symplectic structure, which by Darboux's theorem can be seen as the ``local version'' of this question. Work of Cieliebak and Eliashberg \cite{CE-qconvex} shows that if $Y$ is an embedded closed hypersurface in $\arr^{2m}$ then, as we review in Section \ref{introsec2} below, $Y$ is isotopic to a contact type hypersurface whenever the domain $W\subset \arr^{2m} = \cee^m$ bounded by $Y$ satisfies the complex-analytic condition of {\it rational convexity}. When $2m>4$, Cieliebak-Eliashberg show that the latter is essentially a topological condition on $W$, and many examples can be obtained---thus the local theory of contact type hypersurfaces is quite rich. Here we will focus on the case $\dim(X) = 4$, where the local theory appears to be much more constrained. In fact, the following conjecture appears in the work of Weimin Chen:

\begin{conjecture}[\cite{weimin}, Conjecture 5.1]\label{chenconj} If $Y\subset \arr^4$ is a hypersurface with $b_1(Y) = 0$, and the domain $W$ bounded by $Y$ is rationally convex, then $Y$ is diffeomorphic to $S^3$.
\end{conjecture}

Indeed, no examples of contact type hypersurfaces in $\arr^4$ having vanishing first Betti number are known (assuming rational convexity or not), except those diffeomorphic to $S^3$. Direct evidence for Conjecture \ref{chenconj} was obtained in \cite{MTctype}, where the authors introduced an obstruction for a 3-manifold to arise as a contact type hypersurface in $\arr^4$ and used it to rule out various 3-manifolds including all naturally-oriented Brieskorn homology spheres.

The goal of this work is to begin to consider Question \ref{CTquestion} in a broader context than $\arr^4$, particularly the case where $X$ is a closed symplectic 4-manifold. We focus on a stronger condition on a hypersurface $Y\subset X$ than the contact type condition, namely that $Y$ arise as the boundary of a {\it Weinstein domain} symplectically embedded in $X$. This condition can be seen as a symplectic analog of rational convexity, and from this point of view the following is a generalization of \cite[Question 1.2]{NemSie}:

\begin{question}\label{genweinsteinquestion} If $(X,\omega)$ is a symplectic $n$-manifold, which smooth $n$-manifolds $W$ admit an embedding in $X$ so that $(W, \omega|_W)$ has the structure of a Weinstein domain?
\end{question}

Theorem 1.5(b) of \cite{CE-qconvex} shows that the two conditions (rational convexity and Weinstein) are equivalent up to isotopy for domains in standard $\arr^{2m}$.  The main technical result of this paper, Theorem \ref{ratconvthm} below, can be seen as proving one direction of a similar equivalence in the context of closed symplectic manifolds. The definition of a Weinstein domain is reviewed in Section \ref{introsec2}; the boundary of a Weinstein domain in $X$ is in particular a contact type hypersurface.

In the following theorem, we say that a closed symplectic 4-manifold $(X,\omega)$ is {\it positive} if 
\[
\langle c_1(X,\omega)\cup [\omega], [X]\rangle >0,
\]
where $c_1(X,\omega)$ denotes the first Chern class of the tangent bundle of $X$ equipped with an $\omega$-compatible almost-complex structure. We may omit either $X$ or $\omega$ from the notation $c_1(X,\omega)$ if no confusion will arise. This condition is quite restrictive: the only minimal positive symplectic 4-manifolds are diffeomorphic to $\cee P^2$, $S^2\times S^2$, or a 2-sphere bundle over a higher-genus surface \cite{liu96, OhtaOno96}, and in particular all such $(X,\omega)$ have $b^+(X) = 1$. If $(X,\omega)$ is positive, then by performing a symplectic blowup in a sufficiently small ball, it is clear that the manifold $X\# \cptwobar$ also admits a positive symplectic structure. By a {\it  rational surface} we will mean a smooth 4-manifold diffeomorphic to $S^2\times S^2$ or to $\cee P^2 \# k \cptwobar$ for some $k\geq 0$; thus any rational surface admits some positive symplectic structure. In fact every symplectic structure on a rational surface diffeomorphic to $S^2\times S^2$ or $\cee P^2 \# k\cptwobar$ with $k\leq 9$ is positive (cf. Proposition \ref{smallposprop}), though in general a rational surface may also admit non-positive structures.


Using obstructions described below, we find significant restrictions for a Brieskorn sphere to bound a Weinstein domain symplectically embedded in a positive symplectic rational surface. Here and throughout, we assume Brieskorn spheres are equipped with their natural orientation as links of singularities, and that this orientation agrees with that induced as the boundary of the convex side of a contact type embedding.

\begin{theorem}\label{nonembthm} Suppose $Y = \Sigma(p_1,\ldots,p_n)$ is a Brieskorn integer homology sphere that embeds as a contact type hypersurface in a positive symplectic rational surface $(X,\omega)$, such that the convex side $(W, \omega|_W)$ is a Weinstein domain. Then:
\begin{enumerate}
\item The intersection form of $W$ is nontrivial, negative definite, and even.
\item The first Chern class $c_1(X,\omega)$ restricts trivially to $W$.
\item The second Betti number of $W$ is $b_2(W) = 4d(Y)>0$, where $d(Y)$ is the Heegaard Floer $d$-invariant.
\end{enumerate}
\end{theorem}

As a consequence we obtain:

\begin{corollary}\label{fancybrieskorncor} Let $(X,\omega)$ be a symplectic manifold diffeomorphic to $S^2\times S^2$ or to $\cee P^2\# k \cptwobar$, for some $k\leq 7$. Then no Weinstein domain symplectically embedded in $X$ has boundary diffeomorphic to any Brieskorn homology sphere.
\end{corollary}

We also obstruct certain families of Brieskorn spheres from bounding a Weinstein domain in {\it any} positive symplectic 4-manifold; see Theorem \ref{suitablebrieskornthm} and Theorem \ref{diagthm}.

In general, however, the bound $k\leq 7$ is sharp in Corollary \ref{fancybrieskorncor}. Indeed, it is well known that the Poincar\'e homology sphere $\Sigma(2,3,5)$ embeds in a rational elliptic surface $E(1) \cong \cee P^2\# 9\cptwobar$ as the (negatively oriented) boundary of a Gompf nucleus, which consists of a neighborhoood of a cusp fiber union a section of the elliptic fibration. This neighborhood can be taken to be symplectically concave (with respect to a symplectic form compatible with the fibration), and its complement $W$ is diffeomorphic to a plumbing of cotangent disk bundles over 2-spheres, plumbed according to the $E_8$ graph \cite{gompf:nuclei}. The latter admits a Weinstein structure that agrees near the boundary with the symplectic structure on the concave neighborhood (using, for example, the fact that the Poincar\'e sphere has just one fillable contact structure), hence the symplectic structure on the nucleus extends this Weinstein structure as a symplectic structure to all of $E(1)$. This shows $\Sigma(2,3,5)$ bounds a Weinstein domain in $\cee P^2\# 9\cptwobar$, and by blowing down the section contained in the nucleus (a symplectic sphere of self-intersection $-1$), we find $\Sigma(2,3,5)$ as the boundary of a Weinstein domain in $\cee P^2\# 8\cptwobar$. 

On the other hand, Theorem \ref{nonembthm} and Corollary \ref{fancybrieskorncor} are strictly {\it symplectic} results. Indeed, for any relatively prime $p,q\geq 2$, Gompf \cite[Theorem 4.1]{gompfemb} has exhibited holomorphic embeddings of a Stein domain $W_{p,q}$ in any chosen complex structure on $\cee P^2\#\cptwobar$, whose boundary is the Brieskorn sphere $\Sigma(p,q,pq+1)$; one anticipates many more such examples as $k$ grows. Corollary \ref{fancybrieskorncor} implies that it is not possible to arrange that these Stein domains are symplectically embedded, for any symplectic structure on the ambient manifold. 

In other words, there are many Brieskorn spheres that abstractly bound Weinstein (actually, Stein) domains with very simple topology---in Gompf's examples, just a 0-handle and a 2-handle---which embed smoothly (actually, holomorphically) in a rational ruled surface diffeomorphic to $\cee P^2\#\cptwobar$, but nevertheless the Brieskorn spheres do not bound any domains therein that are Weinstein with respect to the ambient symplectic structure.

 From the point of view of complex geometry, the results of \cite{BGS} mentioned below imply that Gompf's Stein domains are not rationally convex, regardless of the ambient complex algebraic structure.

\begin{remark} Gompf asks in \cite[Question 4.3]{gompfemb} whether it is possible to embed a Brieskorn sphere in a rational ruled surface as the boundary of a pseudoconvex (i.e. Stein) domain whose intersection matrix is $\langle +1\rangle$. Our results do not address this question, but Theorem \ref{nonembthm}(1) shows that the corresponding question for Brieskorn spheres bounding Weinstein domains, in any positive symplectic rational surface, has a negative answer. 
\end{remark}

The results above follow from a general obstruction for a contact manifold $(Y,\xi)$ to bound a Weinstein domain in a positive symplectic manifold. To state it, recall that Heegaard Floer theory associates to any closed oriented 3-manifold $Y$ a homology group $HF^+(Y)$, and to a positive contact structure $\xi$ on $Y$ an element $c^+(\xi)\in HF^+(-Y)$ where $-Y$ stands for $Y$ with the opposite of the given orientation. Here we take Floer homology with coefficients in $\F = \zee/2\zee$. While $HF^+(Y)$ is always infinite-dimensional over $\F$, it admits a finite-dimensional quotient $HF_{red}(Y)$, the {\it reduced} Heegaard Floer homology. For a contact structure $\xi$ on $Y$ we write $c_{red}(\xi)$ for the image of $c^+(\xi)$ under the quotient $HF^+\to HF_{red}$.

\begin{theorem}\label{obstrthm} Let $(Y,\xi)$ be a connected contact 3-manifold and $(X, \omega)$ a closed symplectic 4-manifold. Assume $(Y,\xi)$ admits an embedding in $X$ as a separating contact type hypersurface, with convex side $(W, \omega|_W)$.
\begin{itemize}
\item[a)] If there exists a connected, embedded symplectic surface $\Sigma \subset X - W$ of genus $g$, whose self-intersection satisfies $\Sigma\cdot \Sigma > \max\{0, 2g - 2\}$, then $c_{red}(\xi) = 0$.
\item[b)] If $(X,\omega)$ is positive, and $(W,\omega|_W)$ admits the structure of a Weinstein domain, then a surface as in (a) exists, possibly after a deformation of $W$ through Weinstein domains.
\end{itemize}
\end{theorem}

For a given 3-manifold $Y$, the constraint $c_{red}(\xi) = 0$ in Theorem \ref{obstrthm} may obstruct some contact structures on $Y$ from bounding Weinstein domains in positive symplectic manifolds but not others. Indeed, this situation arises for certain Brieskorn spheres, giving rise to Theorem \ref{nonembthm} as stated. On the other hand, we can exhibit infinite families of Brieskorn spheres that do not bound Weinstein domains in any positive symplectic manifold: a sample result is as follows.

\begin{theorem}\label{suitablebrieskornthm} Let $p, q\geq 2$ be relatively prime, and let $n$ be a positive integer. Let $Y$ be an oriented 3-manifold diffeomorphic to $\Sigma(p,q, npq+1)$ (for any $n\geq 1)$ or to $\Sigma(p,q,npq-1)$ (for any $n\geq 2$). If $(X,\omega)$ is a positive symplectic 4-manifold, then no Weinstein domain in $X$ has boundary orientation-preserving diffeomorphic to $Y$.
\end{theorem}

Theorem \ref{suitablebrieskornthm} follows from Corollary \ref{suitablecor} and Proposition \ref{suitableprop}.

Returning to Question \ref{CTquestion}, consider the case $X = \cee P^2$ with its essentially unique symplectic structure. Despite the constraints that follow from Corollary \ref{fancybrieskorncor}, there are many more rational homology spheres that appear as contact type hypersurfaces in $\cee P^2$ than seem to arise in $\arr^4$. Motivated by the examples known to the authors, we ask:

\begin{question}\label{weinsteinquestion} Let $Y$ be a contact type hypersurface in $\cee P^2$. Is the convex side of $Y$ necessarily a Weinstein domain?
\end{question}
\begin{question}\label{Lspaceconj}  Let $(Y,\xi)$ be a contact type hypersurface in $\cee P^2$, with $b_1(Y) = 0$. Is it the case that $HF_{red}(Y, \s_\xi) = 0$? 
\end{question}

Here $\s_\xi$ indicates the \spinc structure associated to the contact structure $\xi$. A 3-manifold $Y$ with $b_1(Y) = 0$ and  $HF_{red}(Y,\s) = 0$ for all $\s$ is called an $L$-space; if $Y$ is an integer homology sphere then it carries just one \spinc structure, and therefore the condition in Question \ref{Lspaceconj} is equivalent to $Y$ being an $L$-space in that case. A positive answer to Question \ref{Lspaceconj} would be a significant strengthening of Theorem \ref{obstrthm} in the case of $\cee P^2$, which asserts only that the single element $c_{red}(\xi)$ of $HF_{red}(-Y,\s_\xi)\cong HF_{red}(Y,\s_\xi)$ is trivial (and only under the Weinstein condition).


To add context to Questions \ref{weinsteinquestion} and \ref{Lspaceconj}, we review some of the known examples of contact type hypersurfaces in small rational surfaces. 
\begin{itemize}
\item If $L\subset X$ is a Lagrangian submanifold, then a suitable neighborhood of $L$ is a symplectically embedded Weinstein domain; this is also true if $L$ is allowed to have certain sorts of singularities. Thus, for example, one can find Lagrangian tori in a Darboux chart on $X$, giving rise to contact type hypersurfaces diffeomorphic to $T^3$. In an analogous vein, a construction due to Nemirovski and Siegel \cite{NemSie} exhibits the trivial disk bundle over any orientable surface $\Sigma$ of positive genus as a Weinstein neighborhood of a (singular) Lagrangian in $\cee^2$, and similar with certain disk bundles over non-orientable surfaces. Hence in particular we can find $\Sigma\times S^1$ embedded as a contact type hypersurface  in a Darboux chart in any symplectic 4-manifold. The structure of $HF_{red}(\Sigma\times S^1)$ is known and nontrivial in the relevant \spinc structure when the genus of $\Sigma$ is at least 2 \cite{OS:hfk,JMsurface}, thus a condition on $b_1(Y)$ is necessary for a positive answer to Question \ref{Lspaceconj}.\footnote{It is an elementary observation that a symplectic 4-manifold with $b^+ = 1$ does not contain any embedded orientable Lagrangian surfaces of genus $g>1$; see for example \cite[Remark 1.2]{charette15}. For a more difficult proof of a weaker result, one can also use Theorem \ref{obstrthm} to obstruct such embeddings in positive symplectic manifolds.}
\item A particularly interesting family is the collection of ``Lagrangian pinwheels'' $\L_{p,q}$ studied by Evans and Smith \cite{EvansSmith18}, some of which can be found in $\cee P^2$. The corresponding contact type hypersurfaces are lens spaces $L(p^2,pq-1)$, which certainly have $HF_{red} = 0$. Remarkably, Evans and Smith show that such a pinwheel exists in $\cee P^2$ if and only if $p$ is a member of a triple of positive integers $(x,y,z)$ satisfying the Markov equation $x^2 + y^2 + z^2 = 3xyz$, and in fact it follows from their work that the corresponding lens spaces are exactly the ones arising as contact type hypersurfaces in $\cee P^2$. It was shown in \cite{owensnonsymp} that infinitely many lens spaces not on this list embed smoothly in $\cee P^2$.
\item If a $C\subset X$ is a singular symplectic curve in the sense of \cite{gollastarkston19} (for example, a singular  algebraic curve) having positive self-intersection, then $C$ has a symplectically {\it concave} neighborhood $N(C)$ whose boundary is a contact type hypersurface. This gives a variety of such hypersurfaces $Y\subset \cee P^2$; an illustrative family is that arising from the rational unicuspidal curves whose singularity has only one Puiseaux pair. In this case $Y$ is diffeomorphic to the result of a negative surgery along a negative torus knot (see \cite{FLMN} for the full classification), and is an $L$-space. See Section \ref{largerationalsec} for additional discussion.
\item In general, if $C$ is a rational cuspidal curve in $\cee P^2$ with more than one singularity, then the corresponding contact type hypersurface $Y$ has $b_1(Y) = 0$ but is not an $L$-space. However, it can be shown that the answer to Question \ref{Lspaceconj} is positive in this case: see Proposition \ref{cuspidalprop}. Regarding Question \ref{weinsteinquestion}, it can be shown via standard K\"aher geometry that the complement of such a curve admits a Weinstein (even Stein) structure; see \cite{giroux:remarks} for a symplectic perspective similar to ours.
\item In Section \ref{planarsec} below, we describe a family of contractible $4$-manifolds $W_n$ admitting Weinstein structures that embed symplectically in the small rational surface $\cee P^2 \# 3\cptwobar$ (see Proposition \ref{planarprop}). Necessarily the corresponding contact structure has $c_{red} = 0$, however, the integer homology sphere $Y = \partial W$ is not an $L$-space. In other words, Question \ref{Lspaceconj} has a negative answer if $\cee P^2$ is replaced by a larger symplectic manifold, even a positive one.
\end{itemize}

We note that it is a longstanding open problem whether there exist any irreducible integer homology sphere $L$-spaces other than $S^3$ and the Poincar\'e sphere, up to orientation. In particular, a positive answer to Question \ref{Lspaceconj} would exclude a great many integer homology spheres that bound Stein domains holomorphically embedded in $\cee P^2$ (or even $\cee^2$) from being isotopic to contact type hypersurfaces---and hence the corresponding domains from being rationally convex up to deformation.

We close this section with a couple of additional remarks.

\begin{remark}\label{introremark}
\begin{enumerate}
\item It is not hard to see that any Brieskorn sphere, indeed any suitably-oriented Seifert 3-manifold, admits a contact type embedding in some positive symplectic rational surface $X$: see Section \ref{largerationalsec}. In the construction we describe, the convex side of such an embedding is not Weinstein, and indeed typically cannot be in light of the results above. In some cases, it is possible to modify the symplectic form on the convex side in such a way that it becomes Weinstein; in this situation the resulting symplectic structure on $X$ is non-positive except possibly in the circumstances allowed by Theorem \ref{obstrthm}.
\item For $Y$ to admit a separating contact type embedding in any $(X,\omega)$, certainly $Y$ must support a strongly symplectically fillable contact structure $\xi$. Granted this, if $W$ is a strong symplectic filling of $(Y,\xi)$, then by adding a ``symplectic cap'' \cite{eliashbergremarks,etnyrefillings,GScaps} one can construct a closed symplectic 4-manifold $X = W \cup_Y Z$ containing $Y$ as a separating contact type hypersurface. Thus any strongly fillable $(Y,\xi)$ admits a separating contact type embedding in {\it some} $(X,\omega)$, and the interest of Question \ref{CTquestion} lies in having fixed $X$ ahead of time. Alternatively, one could fix $Y$ and ask for possible constraints on $X$: in particular one might ask what is the ``simplest''  4-manifold into which $Y$ admits a contact type embedding. Analogous remarks apply in the context of Question \ref{genweinsteinquestion}, where a ``cap'' for a Weinstein manifold (equivalently by \cite{CE}, a Stein manifold) was first obtained by Lisca-Mati\'c \cite{LM}.
\end{enumerate}
\end{remark}

\subsection{Additional context and main structural theorem}\label{introsec2}
Here we review the notion of rational convexity and its relationship with symplectic topology, and state the central result leading to Theorem \ref{obstrthm}.

A subset $W$ in a smooth complex algebraic variety $X$ is {\it rationally convex} if through every point $p\in X - W$ there passes an algebraic hypersurface that is disjoint from $W$. While rational convexity is a classical topic in complex analysis, a variety of results in recent decades indicate that rational convexity has close connections with symplectic geometry. An instance of this connection is a result of Duval-Sibony \cite{DS} and Nemirovski \cite{Nem}: it states that  a domain $W\subset\cee^n$ is rationally convex if and only if it admits a Stein structure $(W, \varphi)$ whose associated K\"ahler form extends as a K\"ahler form to all of $\cee^n$.  Here a ``domain'' means a compact subset that is the closure of a connected open subset with smooth boundary, and a domain is Stein if it can be written as the sub-level set of an exhausting, strictly plurisubharmonic function $\varphi$ defined in a neighborhood of the domain. In this case the form $\omega_\varphi = - dd^c\varphi$ is the associated symplectic (K\"ahler) form. 
 
Recall that a {\it Weinstein domain} is a tuple $(W, \omega, \varphi, V)$ where $W$ is a compact manifold with boundary, $\omega$ a symplectic form on $W$, $\varphi: W\to [0,c]$ is a Morse function on $W$ such that  $\partial W = \varphi^{-1}(c)$ and $\varphi$ has no critical points on $\partial W$, and $V$ is a complete Liouville vector field on $W$ that is gradient-like for $\varphi$ (we use the terminology of \cite{CE}). As mentioned previously, Cieliebak-Eliashberg proved that a Stein domain $W\subset\cee^n$ is rationally convex if and only if, after an isotopy, it admits a Weinstein structure that is symplectically embedded in $(\arr^{2n}, \omega_{std})$.

The characterization of rational convexity by Duval-Sibony and Nemirovsky was extended to projective complex manifolds by Boudreaux, Gupta and Shafikov \cite{BGS}: they showed that a Stein domain $W\subset X$ in such a manifold is rationally convex if and only if an associated K\"ahler form $\omega_\varphi$ on $W$ extends to all of $X$ as a K\"ahler form $\omega_X$ having integral periods. From a symplectic point of view, a Stein domain is a particular kind of Weinstein domain, and the result of Boudreaux-Gupta-Shafikov implies that rational convexity of (the Stein domain) $W\subset X$ is equivalent to the condition that (the associated Weinstein domain) $W$ be symplectically embedded in $(X, \omega_X)$ for appropriate $\omega_X$.

Our main result can be seen as a purely symplectic version of one direction of the result of \cite{BGS}: essentially, that Weinstein domains in closed symplectic manifolds satisfy a version of rational convexity (in particular, the ambient manifold need not be algebraic nor even admit an integrable complex structure). To give the statement, note that if $(W, \omega, \varphi, V)$ is a Weinstein domain then the flow $\Phi_{-t}$ of $V$ is defined for all $t\geq 0$, and we write $W_t$ for the image $\Phi_{-t}(W)$. Observe that $W_t\subset W$ is again a Weinstein domain (with the restrictions of $\omega$ and $V$) that is conformally symplectomorphic to $W$ by $\Phi_{-t}$. Throughout the paper, if we refer to a Weinstein domain contained in a larger symplectic manifold, we aways assume that the Weinstein domain is equipped with the symplectic structure induced from the ambient manifold.

While our applications in Section \ref{introsec1} concern dimension 4, the following holds in all dimensions.

\begin{theorem}\label{ratconvthm} Let $(X,\omega)$ be a closed $2n$-dimensional symplectic manifold such that the class $\frac{1}{2\pi}[\omega]$ is integral, and $W\subset X$ a Weinstein domain in $X$. Let $p\in X - W$. Then for any sufficiently large integer $k$, there exists a smooth, codimension-2 symplectic submanifold $\Sigma$ containing $p$, with homology class $[\Sigma]$ Poincar\'e dual to the class $\frac{k}{2\pi}[\omega]$, such that $\Sigma \subset X - W_t$ for all sufficiently large $t$.
\end{theorem}

Thus, except for the requirement that the domain $W$ be ``contracted'' (that is, replaced by $W_t$ for some $t>0$ that may depend on $p$), we see that Weinstein domains $W\subset X$ are ``symplectically rationally convex.'' We do not know if the contraction of $W$ is truly necessary, but as our applications are insensitive to this detail we do not pursue the question. Note that the requirement that the class $\frac{1}{2\pi}[\omega]$ be integral is easy to arrange by a small perturbation of $\omega$ followed by a dilation, and also does not restrict our applications. In fact, the existence of symplectic submanifolds in the complement of a Weinstein domain holds without the integrality assumption on $\frac{1}{2\pi}[\omega]$, cf. Section \ref{nonintsec}.

The proof of Theorem \ref{ratconvthm} uses approximately holomorphic techniques introduced by Donaldson \cite{don96}, and in particular the submanifold $\Sigma$ dual to $\frac{k}{2\pi}[\omega]$ is obtained as the zero locus of one of a suitable sequence of sections $\{\sigma_k\}$ of tensor powers $L^{\otimes k}$ of a line bundle with $c_1(L ) = \frac{1}{2\pi}[\omega]$. Such a submanifold $\Sigma$ is often referred to as a {\it Donaldson divisor}. Our proof is an adaptation of an argument given by Auroux, Gayet and Mohsen \cite{AGM01}, who proved that an isotropic submanifold $L\subset X$ of a symplectic manifold is symplectically rationally convex---itself the symplectic analog of a result of Guedj \cite{guedj} in the algebraic setting.

It is natural to ask about the hypotheses of Theorem \ref{ratconvthm}, particularly the condition that $W\subset X$ be Weinstein. Clearly if a submanifold $\Sigma$ as in the statement is contained in $X - W$, then $\omega|_W$ must be exact: thus a natural weaker condition is that $W$ be merely a Liouville domain. This requires the existence of a Liouville vector field $V$ on $W$, directed out of $\partial W$, but not that $V$ be gradient-like for a Morse function on $W$. We do not know whether Theorem \ref{ratconvthm} as stated holds in this case. However, if one requires that the submanifold $\Sigma$ be obtained as a Donaldson divisor, then it certainly does not, which we can see as follows.

Consider a 4-dimensional Liouville domain $W$ with disconnected boundary $\partial W = Y_0\sqcup Y_1$. (The first examples of such manifolds were found by McDuff \cite{mcduff91}.) Each of $Y_0$ and $Y_1$ inherits a contact structure, which admit concave caps $V_0$ and $V_1$ by the constructions of \cite{eliashbergremarks,etnyrefillings,GScaps}. Indeed, these constructions are quite flexible, and one can arrange that $V_0$ and $V_1$ have substantial topology---in particular, we can ensure that neither of their symplectic forms is exact. The union $X = V_0 \cup_{Y_0} W\cup_{Y_1} V_1$ is  a closed symplectic 4-manifold containing $W$ as a symplectically embedded Liouville domain, and by the conditions on $V_0$ and $V_1$, any symplectic submanifold $\Sigma\subset X$ with $[\Sigma]$ Poincar\'e dual to $\frac{k}{2\pi}[\omega]$ must meet both $V_0$ and $V_1$. However, it is known that if $k$ is sufficiently large and  $\Sigma$ is obtained as a Donaldson divisor, then $\Sigma$ is connected \cite[Proposition 39]{don96}, hence must also intersect $W$.

\subsection{Organization and acknowledgements}

We give the proof of Theorem \ref{nonembthm}, Corollary \ref{fancybrieskorncor}, and Theorem \ref{obstrthm} in Section \ref{proofsec1}, along with some additional corollaries of the latter. Section \ref{appsec} illustrates the results with examples and applications, including the proof of Theorem \ref{suitablebrieskornthm}. The proof of Theorem \ref{ratconvthm} is in Section \ref{proofsec2}. 

The authors would like to thank Denis Auroux, T.\ J.\ Li, Tye Lidman, Sam Lisi, and Laura Starkston for helpful communications and conversations related to this work. We are also grateful to Blake Boudreaux, Purvi Gupta, and Rasul Shafikov for stimulating discussions at the Banff International Research Station, and to the BIRS center for providing an environment encouraging to these interactions. Thomas Mark was supported in part by grants from the Simons Foundation (numbers 523795 and 961391). B\"{u}lent Tosun was supported in part by grants from National Science Foundation (DMS2105525 and CAREER DMS 2144363) and the Simons Foundation (2023 Simons Fellowship). He also acknowledges the support by the Charles Simonyi Endowment at the Institute for Advanced Study.

\section{Proof of Theorems \ref{nonembthm} and \ref{obstrthm}}\label{proofsec1}

\subsection{Obstruction to Weinstein boundaries}
First we prove Theorem \ref{obstrthm} and some additional corollaries, assuming Theorem \ref{ratconvthm}. 

\begin{proof}[Proof of Theorem \ref{obstrthm}(a)] First, it follows from \cite{Gay2003} that since $\Sigma$ has positive self-intersection, a small tubular neighborhood $D$ of $\Sigma$ is symplectically concave: that is, there is a Liouville field defined near $\partial D$ and pointing transversely into $D$. Hence, the manifold
\[
Z = X - (\Int(W) \sqcup \Int(D))
\]
is a strong symplectic cobordism from $(Y,\xi)$ to $S = -\partial D$ with an induced contact structure $\eta$. 

Second, note if $g\geq 1$ then $S$ is diffeomorphic to a circle bundle over $\Sigma$ having degree $-\Sigma\cdot \Sigma<2-2g$ (the sign is because we have reversed orientation). In other words, $S$ is a circle bundle whose degree is ``large'' in absolute value compared to the genus $g(\Sigma)$. In this situation, a result of Ozsv\'ath and Szab\'o \cite[Theorem 5.6]{OS:integersurgery} shows that $HF_{red}(-S) = 0$, or equivalently  that the natural map $HF^\infty(-S)\to HF^+(-S)$ is surjective. The same conclusion holds for $g = 0$, since in this case $S$ is $S^3$ or a lens space, which also have $HF_{red} = 0$. Necessarily then, the contact invariant $c(\eta) \in HF^+(-S)$  is in the image of $HF^\infty\to HF^+$.

Finally, recall that the contact invariant is natural under strong symplectic cobordisms. This is a result of Echeverria \cite{mariano}, who showed that if $Z: Y\to S$ is such a cobordism, then the homomorphism in monopole Floer homology $\hmto(-S)\to \hmto(-Y)$ induced by $Z$ with the \spinc structure determined by $\omega$ carries $c(\eta)$ to $c(\xi)$. Recall that monopole Floer homology of a 3-manifold $M$ is isomorphic to Heegaard Floer homology via isomorphisms
 \[
 \hmfrom(M) \cong HF^-(M) \qquad \hmbar(M)\cong HF^\infty(M) \qquad \hmto(M)\cong HF^+(M)
 \]
 that respect the decompositions along \spinc structures and also commute with the maps in the long exact sequences relating the groups \cite{KLT1, CGH1}. Moreover, the isomorphism preserves the contact invariant \cite{CGH3}. Hence the monopole contact invariant lies in the image of $\hmbar\to \hmto$ if and only if the Heegaard Floer contact invariant is in the image of $HF^\infty\to HF^+$. Echeverria's result, together with the equivariance of homomorphisms induced by cobordisms, implies that the former property holds for the monopole contact invariant of $(Y,\xi)$, and hence so does the latter property for the Heegaard Floer contact invariant. This completes the proof of Theorem \ref{obstrthm}. Note that our somewhat circuitous logic is necessary because the isomorphisms between monopole and Heegaard Floer groups are not currently known to commute with maps induced by cobordisms, so strictly Echeverria's result is not known in the Heegaard Floer setting. While Theorem \ref{obstrthm} could equivalently be formulated in terms of monopole Floer homology, the results we need from \cite{MTctype} for our applications here are derived using calculational tools in Heegaard Floer theory and we stick with the latter language for that reason.\end{proof}

\begin{proof}[Proof of Theorem \ref{obstrthm}(b)] Assume that $(X,\omega)$ has the property that $\frac{1}{2\pi}[\omega]$ is an integral cohomology class. We show how to remove this condition in Section \ref{nonintsec}.

Let $W\subset X$ be a Weinstein domain with boundary the contact manifold $(Y, \xi)$. By Theorem \ref{ratconvthm} we may assume that for a sufficiently large integer $k$ and after contracting $W$ suitably, there is a symplectic surface $\Sigma$ embedded in the complement of $W$, representing the class Poincar\'e dual to $\frac{k}{2\pi}[\omega]$. Note that, as a Donaldson divisor, $\Sigma$ is connected \cite[Proposition 39]{don96}.

Since $\Sigma$ is symplectic, we have the adjunction formula 
\[
\langle c_1(X), \Sigma\rangle = 2 - 2g(\Sigma) + \Sigma\cdot \Sigma,
\]
In particular, if $(X,\omega)$ is a positive symplectic manifold then since $[\Sigma]$ is a positive multiple of $\frac{1}{2\pi}[\omega]$, the left side of the equation above is positive. Since clearly $\Sigma\cdot\Sigma >0$,  it follows that $\Sigma \cdot\Sigma > \max\{0,2g(\Sigma) - 2\}$.
\end{proof}

\begin{remark} It is known by the work of many authors that if $\omega$ is a positive symplectic structure on a 4-manifold $X$, then $\omega$ is a K\"ahler form: see \cite[Proposition 1.10]{LiNing}, and \cite{Lisurvey08} for additional references. Using this it is also possible to deduce Theorem \ref{obstrthm}(b) from the results of Boudreaux-Gupta-Shafikov \cite{BGS} in K\"ahler geometry. 
\end{remark}

\subsection{Suitable 3-manifolds}\label{suitablesec}

Observe that Theorem \ref{obstrthm} applies to a particular contact structure $\xi$ on $Y$, providing an obstruction for $(Y,\xi)$ to arise as a Weinstein boundary. If, as in the context of Question \ref{CTquestion}, one wishes to use Theorem \ref{obstrthm} to exclude a given 3-manifold $Y$ from arising this way one is {\it a priori} obliged to consider all possible contact structures on $Y$, or at least all the Weinstein fillable ones. Let us say that a 3-manifold $Y$ is {\it suitable} if every Weinstein fillable contact structure $\xi$ on $Y$ has the property that $c_{red}(\xi)\neq 0$. We will say $Y$ is {\it strongly suitable} if the same condition holds for all strongly filllable contact structures. Theorem \ref{obstrthm} immediately gives:

\begin{corollary}\label{nosuitablepositive} No suitable 3-manifold $Y$ is the oriented boundary of a Weinstein domain in a positive symplectic manifold.\hfill$\Box$
\end{corollary}

We say a rational surface is {\it small} if it is diffeomorphic to $S^2 \times S^2$ or to $\cee P^2 \# k \cptwobar$ with $k\leq 9$. The the following is a consequence of long-established results in the literature:
\begin{proposition}[see \cite{salamonsurvey} and references cited therein]\label{smallposprop} Every symplectic structure on a small rational surface is positive.\hfill$\Box$
\end{proposition}

Thus:

\begin{corollary}\label{nosuitablesmall} No suitable 3-manifold bounds a Weinstein domain symplectically embedded in any small rational surface.\hfill$\Box$
\end{corollary}

Checking that a given $Y$ is (strongly) suitable is typically difficult, as it appears to require a classification of fillable contact structures on $Y$ and a calculation of $c_{red}$ for each. However, we can exhibit some infinite families of suitable 3-manifolds. We focus on the case that $Y$ is a Brieskorn homology sphere $Y = \Sigma(p_1,\ldots, p_n)$. Recall that this means $Y$ arises as the link of a certain normal complex surface singularity, and in particular $Y$ is diffeomorphic to the boundary of a negative-definite plumbed 4-manifold $X_\Gamma$ (a minimal good resolution of the singularity; see \cite{MTctype} and the references cited there for additional detail). Note that a Brieskorn sphere inherits a preferred orientation as the boundary of $X_\Gamma$, and all our results assume this orientation. Furthermore, every Brieskorn sphere admits a Weinstein fillable contact structure, often many such.

A general result regarding suitability of Brieskorn spheres was obtained by the authors in \cite{MTctype}:
\begin{theorem}\cite[Theorem 5.1]{MTctype}\label{diagthm} If the intersection form of $X_\Gamma$ is diagonalizable over the integers, then every strongly fillable contact structure $\xi$ on $\Sigma(p_1,\ldots, p_n) = \partial X_\Gamma$ has $c_{red}(\xi) \neq 0$. In other words, $\Sigma(p_1,\ldots, p_n)$ is strongly suitable.
\end{theorem}

As cases of Theorem \ref{diagthm}, we have:
\begin{corollary}\label{suitablecor} Any Brieskorn sphere whose Heegaard Floer $d$-invariant vanishes is strongly suitable. This includes:
\begin{enumerate}
\item  Any Brieskorn sphere $Y$ that bounds a rational homology ball, for example the Brieskorn spheres bounding contractible 4-manifolds obtained by Casson-Harer \cite{CH}.
\item The Brieskorn sphere $\Sigma(p,q,npq +1)$, for relatively prime integers $p,q\geq 2$ and integer $n\geq 1$.
\end{enumerate}
\end{corollary}

\begin{proof} The Heegaard Floer $d$-invariant $d(Y)\in \zee$ of an oriented integer homology sphere $Y$ has the property that if $d(Y) = 0$ then the intersection form of any negative-definite 4-manifold with boundary $Y$ is diagonalizable: this follows from \cite[Corollary 9.7]{OS:grading} and Elkies' theorem (see \cite[Theorem 1]{Elkies} or \cite[Theorem 9.5]{OS:grading}). In particular, for a Brieskorn sphere $Y$ with $d(Y) = 0$, the corresponding plumbed manifold $X_\Gamma$ has diagonalizable intersection form, and therefore $Y$ is strongly suitable. Furthermore, any homology 3-sphere that bounds a rational homology ball has $d(Y) = 0$ \cite{OS:grading}, which gives the first statement.

For the second claim, it was shown by N\'emethi  \cite[Section 5.6.2]{nemethi} (see also \cite[Equation (1)]{Tweedy2013}) that the $d$-invariant of $\Sigma(p,q,npq+1)$ is zero. (Note that some members of this family bound contractible 4-manifolds, others do not, and for some it is unknown.)
\end{proof}

%
%

Theorem \ref{diagthm} is not a characterization of suitability, as we see from Proposition \ref{suitableprop} below. Note  that the Poincar\'e sphere $\Sigma(2,3,5)$ is not suitable, since $HF_{red}(-\Sigma(2,3,5)) =0$. Additional examples of non-suitable Brieskorn spheres include the family $\Sigma(p, q, pq-1)$ for $	p$ even and $q=pk+1$ with $k$ odd; see the discussion in Section \ref{nsuitable}.

\subsection{Constraints on Brieskorn spheres in rational surfaces}

\begin{proof}[Proof of Theorem \ref{nonembthm}]  
Recall that for a contact structure $\xi$ on a rational homology sphere $Y$, the contact invariant $c^+(\xi)$ lies in the group $HF^+_{h(\xi)}(-Y)$, where the grading $h(\xi)$ is equal to 
\begin{equation}\label{contactgrading}
h(\xi) = -\ts\frac{1}{4}(c_1^2(X, J) - 3\sigma(X) - 2e(X)) - \frac{1}{2}.
\end{equation}
Here $(X,J)$ is any compact almost-complex 4-manifold with $\partial X = Y$ such that $\xi$ is positively $J$-invariant. 

Now, the Floer homology $HF^+(-Y)$ is a module over the polynomial ring $\F[U]$ (where $U$ carries degree $-2$), and for any contact structure the class $c(\xi)$ lies in $\ker(U)$. Moreover, for an integer homology sphere $Y$ one can find an isomorphism $HF^+(-Y)\cong \T^+_d \oplus M$, where $M$ is isomorphic to a sum of cyclic modules of the form $\F[U]/U^k$, and $\T^+ \cong \F[U, U^{-1}]/ U\F[U]$. The subscript $d$ indicates the smallest degree of a nonzero element of $\T^+$, and by definition is the $d$-invariant  $d(-Y)$. In particular $\ker(U)\cap \T^+_d$ is one-dimensional and lies in degree $d$. Finally, the subgroup $\T^+_d$ is characterized as the kernel of the projection $HF^+(-Y)\to HF^+_{red}(-Y)\cong M$, and therefore the condition that $c_{red}(\xi) =0$ is equivalent to the claim $c(\xi)\in \ker U \cap \T^+_d$. 

It follows that if $h(\xi) \neq d(-Y)$, then $c(\xi)$ does not lie in the kernel of $HF^+\to HF_{red}$, or equivalently $c_{red}(\xi) \neq 0$.

A property of the $d$-invariant is that $d(Y) \geq 0$ if $Y$ bounds a negative-definite 4-manifold, and similarly, $d(Y)\leq 0$ if $Y$ is the oriented boundary of a positive-definite 4-manifold. Since a Brieskorn sphere $Y$ 
always bounds a negative-definite plumbed 4-manifold $X_\Gamma$, we have $d(Y)\geq 0$, and we have already pointed out that if $d(Y) = 0$ then the intersection form of $X_\Gamma$ is diagonalizable.

Let $Y$ be a Brieskorn homology sphere embedded in a positive symplectic rational surface $(X,\omega)$ as the boundary of a Weinstein domain. By the remarks above and Theorem \ref{obstrthm}, we must have $d(Y)>0$. Then $Y$ separates $X$ into two components as $X = X_A \cup_Y X_B$, neither of which can be a homology ball (because if $Y$ bounds a homology ball then $d(Y) = 0$). There is a corresponding nontrivial decomposition of the integer homology $H_2(X) = A \oplus B$ where $A = H_2(X_A)$ and $B = H_2(X_B)$. This decomposition is orthogonal with respect to the intersection form $Q_X$ on $H_2(X)$, and $Q_X$ restricts to unimodular symmetric bilinear forms $Q_A$ and $Q_B$ on the factors. By additivity of the signature, exactly one of these intersection forms is negative definite: let us assume it is $Q_A$. 

We claim that $Y$ must be oriented as the boundary of $X_A$, and hence $X_A = W$ is the convex side of the contact type embedding and a Weinstein domain. Indeed, if the opposite were true then we would have $-Y = \partial X_A$, which would imply $-d(Y) = d(-Y)\geq 0 $ since $X_A$ is negative definite, contradicting our condition $d(Y)>0$. Thus the intersection form of the Weinstein domain bounded by $Y$ is nontrivial and negative definite. 

Let $\xi$ be the contact structure induced on $Y$ by $X_A$. According to \cite{CE}, we can equip $X_A$ with a Stein structure $J$ deformation equivalent to the Weinstein structure. A theorem of Ozsv\'ath, Stipsicz and Szab\'o \cite{OSS:planar} states that if $(Y,\xi)$ is a contact manifold admitting a Stein filling $(X_A, J)$, such that the Euler class $e(\xi)$ is trivial and the Chern class $c_1(J)$ is nontrivial, then $c_{red}(\xi)\neq 0$. Triviality of $e(\xi)$ is automatic for our integral homology sphere $Y$, so we see that nonvanishing of $c_1(X_A)$ violates Theorem \ref{obstrthm}. Hence the restriction of $c_1(X)$ to $X_A$ must vanish. Since $c_1(X_A)$ is always a characteristic vector for $Q_A$, its vanishing implies that $Q_A$ is an even form.

Now consider the grading $h(\xi)$ of the contact invariant. We can use $X_A$ with the complex structure $J$ in the formula \eqref{contactgrading}, which reduces easily to 
\[
\ts h(\xi) = -\frac{1}{4} b_2(X_A).
\]
We have seen that a necessary condition for $c_{red}(\xi) = 0$ is that $h(\xi) = d(-Y)$. It follows that $b_2(X_A) = 4d(Y)$. 
\end{proof}

\begin{proof}[Proof of Corollary \ref{fancybrieskorncor}] This follows from Theorem \ref{nonembthm} because if a Brieskorn sphere $Y$ bounds a Weinstein domain $W$ in a rational surface $X$, then since $W$ is negative definite we have $b^-(X) \geq b^-(W) = b_2(W) = 4d(Y) \geq 8$, where the last inequality holds since $d(Y)$ is always an even integer and we have seen it must be strictly positive. Alternatively, we note that the smallest rank of a negative definite, unimodular, even, symmetric bilinear form over $\zee$ is $8$.
\end{proof}

\section{Applications and examples}\label{appsec}

\subsection{Non-suitable Brieskorn spheres}\label{nsuitable} In this section we complete the proof of Theorem \ref{suitablebrieskornthm} and provide some explicit examples of non-suitable Brieskorn spheres.

\begin{proposition}\label{suitableprop} Consider the family of Brieskorn homology spheres $\Sigma(p,q, npq-1)$, for $p, q>1$ relatively prime, and $n\geq 1$.
\begin{enumerate}
\item If $n>1$, then $\Sigma(p,q,npq-1)$ is strongly suitable, hence bounds no Weinstein domain in a positive symplectic 4-manifold.
\item For $n =1$, there is at most one fillable contact structure $\xi_0$ on $\Sigma(p,q,pq-1)$ for which $c_{red}(\xi_0) = 0$. 
\item The condition $c_{red}(\xi_0) = 0$ occurs for infinitely many pairs $(p,q)$, so for such pairs $\Sigma(p, q, pq-1)$ is not suitable. 
\end{enumerate}  
\end{proposition}

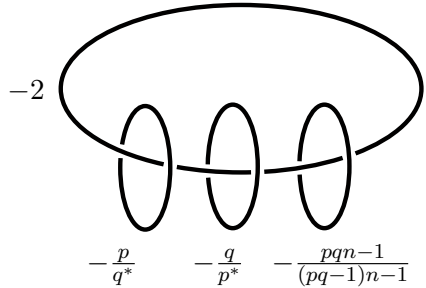
\begin{figure}[t]
\includegraphics[width=2in]{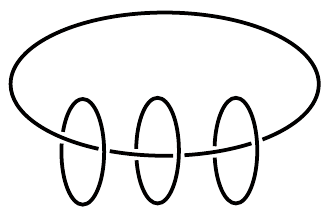}
\put(-130,-10){$-\frac{p}{q^*}$}
\put(-90,-10){$-\frac{q}{p^*}$}
\put(-60,-10){$-\frac{pqn-1}{(pq-1)n -1}$}
\put(-160,55){$-2$}
\caption{\label{surgdiag}Surgery diagram for $\Sigma(p,q,npq-1) = M(-2; \frac{q^*}{p}, \frac{p^*}{q}, \frac{(pq-1)n-1}{pqn-1})$}
\end{figure}

It seems an interesting question to determine, given $(p,q,r)$, what is the smallest value of $k$ so that $\Sigma(p,q,r)$ bounds a Weinstein domain in {\it some} symplectic structure on $\cee P^2 \# k\cptwobar$.  

To begin the proof, we first note that $\Sigma(p, q, pqn-1)$ is the result of $-\frac{1}{n}$ surgery on the negative torus knot $T_{p,-q}$ where $1<p<q$ relatively prime and $n\geq 1$. We can also describe this manifold as the small Seifert fibered space $M(-2; \frac{q^*}{p}, \frac{p^*}{q}, \frac{(pq-1)n-1}{pqn-1})$ where $p^*$ and $q^*$ are multiplicative inverses of $p$ and $q$, respectively: specifically, we require $0<p^*< q$ and $pp^* = 1$ mod $q$, and $0<q^*< p$ and $qq^* = 1$ mod $p$ (see, for example, \cite[Lemma 4.4]{OwensStrle12}, noting the orientation reversal and negative reciprocals in Seifert invariants). A surgery diagram for $\Sigma(p,q,npq-1)$ is shown in Figure \ref{surgdiag}. We also note that this Seifert fibered space arises as the boundary of a negative definite plumbed $4$-manifold $X_{\Gamma_{p,q,n}}$, and we start by describing the plumbing diagram for this manifold. 

Recall that for a sequence of integers $a_1,\ldots, a_m$, the associated ``negative'' continued fraction is 
\[
[a_1,a_2,\ldots, a_m]^-=a_1-
\cfrac{1}{a_2- \cfrac{1}{\ddots  -\cfrac{1}{a_m}}}.
\] 
A negative rational number $-\frac{r}{s}$ has a unique expression $-\frac{r}{s} = [-a_1,\ldots, -a_m]^-$, if one requires $a_j\geq 2$ for each $j\geq 2$. Furthermore, the $-\frac{r}{s}$ Dehn surgery on an unknot is equivalent to surgery on a simple ``chain'' of unknots with coefficients $-a_1,\ldots, -a_m$. This procedure allows us to convert the surgery diagram of Figure \ref{surgdiag} into an integer surgery diagram on a link of unknots: the link is described by a starshaped graph having a single vertex of valence 3 and three ``legs'', where vertices correspond to unknots and edges to simple linking. The surgery coefficients are $-2$ on the central vertex, and on the vertices in the legs they are given by the negative continued fraction coefficients of $-\frac{p}{q^*}$, $-\frac{q}{p^*}$ and $-\frac{pqn-1}{(pq-1)n -1}$.

It is a straightforward exercise to see that the continued fraction for $-\frac{pqn-1}{(pq-1)n-1}$ is
\[
-\frac{pqn-1}{(pq-1)n-1} = [\underbrace{-2, \cdots,-2}_{pq-2}, -3,\underbrace{ -2, \cdots,-2}_{n-2}]^-
\]
if $n>1$ and 
\[
-\frac{pq-1}{pq-2} = [\underbrace{-2, \cdots,-2}_{pq-2}]^-
\]
if $n=1$. 

We next follow a standard procedure of converting the handle diagram for $X_{\Gamma_{p,q,n}}$ into a Legendrian surgery diagram of a link of Legendrian unknots $L_i$. To do so, first realize each unknot in the diagram for $X_{\Gamma_{p,q,n}}$ as a standard Legendrian unknot with $tb=-1$ and $rot=0$. Then stabilize each component a number of times according to the corresponding weight of the vertex (or equivalently the entry in the continued fractions mentioned above): namely, if $-n_i\leq -2$ is the weight of the vertex, then then we stabilize $L_i$ a total of $n_i-2$ times to yield a Legendrian unknot with $tb = -n_i +1$. Attaching a Stein 2-handle along each of these Legendrians yields a Stein structure $J$ on $X_{\Gamma_{p,q,n}}$, hence a fillable contact structure on $\Sigma(p,q,pqn-1)$. Moreover, after fixing the orientations, we have  
\[
\langle c_1(J), S_i\rangle=rot(L_i),
\] 
where $S_i$ is (the homology class of) a capped off Seifert surface for $L_i$. We note that each stabilization of $L_i$ can be chosen to be positive or negative, and changes $\rot(L_i)$ by $\pm 1$ correspondingly. This yields $n_i-1$ possible rotation numbers $rot(L_i)$. In particular, if $n_i$ is odd, then the corresponding rotation number is necessarily nonzero. We organize all possible rotation numbers $rot(L_i)$ into a vector $\bf rot$. A result of Lisca-Matic \cite[Theorem 1.2]{LM} determines that different Chern classes (or equivalently different $\bf rot$ vectors) must induce non-isotopic contact structures on the boundary. In \cite{EtnyreMinTosunVarvarezosPre} it is proven that all strongly fillable (indeed all tight) contact structures on $\Sigma(p,q,pqn-1)$ come from this procedure. In particular, the vector $\bf rot$ enumerates all fillable structures on $\Sigma(p,q,pqn-1)$. With this understood, we can proceed with the proof of Proposition \ref{suitableprop}.

\begin{proof}[Proof of Proposition \ref{suitableprop}(1)]
In this case, as we assume $n>1$, we have \[-\frac{pqn-1}{(pq-1)n-1}=[\underbrace{-2, \cdots,-2}_{pq-2}, -3,\underbrace{ -2, \cdots,-2}_{n-2}].\] In particular, because of the entry $-3$, regardless of the weights on the other two legs, we have $\bf rot\neq \bf 0$. In other words, corresponding to any fillable contact structure on $\Sigma(p,q,pqn-1)$, we have a Stein structure on $X_{\Gamma_{p,q,n}}$ whose first Chern class is nonvanishing. On the other hand since  $\Sigma(p,q,pqn-1)$ is a homology sphere we have $c_1(\mathfrak{s}_\xi)=0$ for any fillable $\xi$ on $\Sigma(p,q,pqn-1)$. We have already recalled that \cite[proof of Corollary 1.5]{OSS:planar} implies that a contact structure $\xi$ with $c_1(\mathfrak{s}_\xi) = 0$ that admits a Stein filling with nonzero first Chern class must have $c_{red}(\xi)\neq 0$.
\end{proof}

\begin{proof}[Proof of Proposition \ref{suitableprop}(2)] As we have just seen, each strongly fillable contact structure $\xi$ on $\Sigma(p,q,pq-1)$ is filled by a distinct Stein structure $J_{\xi}$ on $X_{p,q,1}$. Moreover, if $c_1(J_\xi) \neq 0$ then $c_{red}(\xi) \neq 0$. Thus the only candidate for a fillable contact structure $\xi_0$ having $c_{red}(\xi_0) = 0$ is one filled by a Stein structure with $c_1(J_{\xi_0}) = 0$, and by the classification in \cite{EtnyreMinTosunVarvarezosPre} there is at most one such contact structure.
\end{proof}

\begin{proof}[Proof of Proposition \ref{suitableprop}(3)] We have seen that $\xi_0$ is characterized as arising from a Legendrian diagram for $X_{p,q,1}$ in which all rotation numbers are zero, or equivalently that the corresponding Chern class vanishes. Note that such a diagram is possible if and only if all coefficients in the continued fraction expansions of $-\frac{p}{q^*}$, $-\frac{q}{p^*}$ and $-\frac{pq-1}{pq-2}$ are even. Since $-\frac{pq-1}{pq-2}=[-2, \cdots,-2]^-$, we focus on the other two fractions. 

Observe that for $q = pk +1$, we have $q^* = 1$ and $p^* = (p-1)k +1$, and the simple expansions
\begin{eqnarray*}
-\frac{p}{q^*} &=& [-p]^-\\
-\frac{q}{p^*} = -\frac{pk+1}{(p-1)k + 1} &=& [\underbrace{-2,\ldots,-2}_{p-1}, -k-1]^-
\end{eqnarray*}
for each $k\geq 1$. Thus, the pairs $(p, pk+1)$ with $p$ even and $k$ odd satisfy the property that all coefficients of the continued fractions are even, and correspond to the family $\Sigma(p, pk+1, p(pk+1) - 1)$ of Brieskorn spheres. A plumbing diagram for the corresponding $X_{\Gamma}$ is in Figure \ref{PG}.

\begin{figure}[t]
\begin{center}
  \includegraphics[width=11cm]{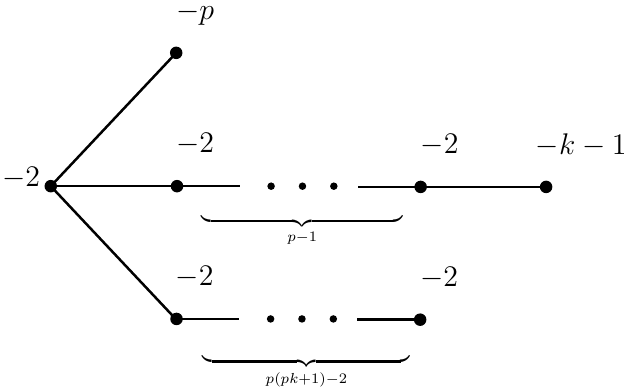}
 \caption{The plumbing graph $\Gamma_{p, pk+1, 1}$. Note $q^*=1$ and $p^*=(p-1)k+1$}
  \label{PG}
\end{center}
\end{figure} 

We claim that no member of this family is suitable. Equivalently, we claim that for $p$ even and $k$ odd, if $\xi_0$ is the contact structure on $\Sigma(p, pk+1, p(pk+1) - 1)$ filled by a Stein structure on $X_{p, pk+1, 1}$ with vanishing Chern class, then $c_{red}(\xi_0) = 0$. A sketch of the argument is as follows.

First, we apply the grading formula \eqref{contactgrading} to $X_{p,pk+1, 1}$, which is negative-definite and has $c_1(J_{\xi_0}) = 0$. Counting handles in the diagram (see Figure \ref{PG}), we find that the contact invariant for $\xi_0$ lies in degree $h(\xi_0) = -\frac{1}{4} p(pk+2)$.

On the other hand, according to Theorem 1 of \cite{Tweedy2013}, we have that for $-Y = S^3_1(T(p,q)) = -\Sigma(p,q,pq-1)$, the $d$-invariant is equal to $-2\alpha_{g-1}$, while under a decomposition $HF^+(-Y) \cong \T_d\oplus M$, the submodule $\ker(U)\cap M$ is contained in a sum of subgroups having degrees $-2\alpha_{g-1+i} + 2i(i-1)$ where $i = 1,\ldots, g-1$. Here the nonnegative integers $\alpha_{g-1}\geq\alpha_g \geq \alpha_{g+1}\geq \cdots$ are defined as follows. Let 
\[
\S_{p,q} = \{ ap + bq \, |\, a,b \in \zee_{\geq 0}\}
\]
be the semigroup generated by $p$, $q$. Then by definition
\[
\alpha_j = \#\{s\in \enn -  \S_{p,q}\,|\, s > j\}.
\]
The genus $g$ is the Seifert genus of the torus knot $T(p,q)$, namely $g = \frac{1}{2}(p-1)(q-1)$. 

The claim $c_{red}(\xi_0) = 0$ in the case $q = kp +1$, with $p$ even and $k$ odd, is therefore a consequence of the following, whose verification we leave to the reader:
\begin{itemize}
\item $\alpha_{g-1} = \frac{1}{8} p(pk + 2)$, so that $h(\xi_0) = d(-\Sigma(p,pk+1, p(pk+1) -1))$.
\item $-2\alpha_{g-1+i} + 2i(i-1) > -2\alpha_{g-1}$ for each $i$. 
\end{itemize}
\end{proof}

\subsection{Rational cuspidal curves}\label{cuspidalcurves}

Here we provide some evidence for a positive answer to Question \ref{Lspaceconj}, by considering the contact type hypersurfaces in $\cee P^2$ arising as the boundaries of concave neighborhoods of rational cuspidal curves. 

Recall that a rational cuspidal curve is a complex algebraic plane curve $C\subset \cee P^2$ that is rational (topologically a sphere) and whose singularities are locally irreducible. If $C$ is such a curve with singular points $z_1,\ldots, z_n$ having (connected) links $K_1,\ldots, K_n$, a neighborhood of $C$ can be given a smooth description as follows. A neighborhood of each singular point is homeomorphic to the cone, in the 4-ball, over the knot $K_k\subset S^3$. Take small balls around each $z_k$ and join them by neighborhoods of arcs in $C$, so that the portion of $C$ lying outside the resulting neighborhood is a smooth disk meeting the boundary in the connected sum $K = K_1\#\cdots \# K_n$. A neighborhood of $C$ is then diffeomorphic to the result of attaching a 2-handle to $B^4$ along $K$ with framing $d^2$, where $d$ is the degree of $C$. Writing $N(C)$ for this neighborhood, it is known that we can choose $N(C)$ to be concave with respect to the symplectic structure. We let $Y = -\partial N(C)$, and $\xi_C$ the positive contact structure on $Y$ arising from the concave structure on $N(C)$. The \spinc structure $\s_{\xi_C}$ is the restriction of the \spinc structure $\s_\omega$ on $\cee P^2$ corresponding to the standard symplectic structure, and we note $c_1(\s_\omega) = c_1(\cee P^2)$. 

\begin{proposition}\label{cuspidalprop} For a rational cuspidal curve $C$ in $\cee P^2$, let $(Y, \xi_C)$ be the corresponding contact type hypersurface as above. Then $HF_{red}(Y, \s_{\xi_C}) = 0$.
\end{proposition}

\begin{proof} We have seen that as an oriented manifold, $Y = - S^3_{d^2}(K)$ where $K = K_1\#\cdots\# K_n$. By the invariance of Floer homology under orientation reversal it suffices to prove the result for the corresponding \spinc structure on $-Y = S^3_{d^2}(K)$. Now, the \spinc structures on this surgery manifold can be enumerated as $\{\s_k\,|\, k\in \zee/d^2\zee\}$ by declaring $\s_k$ to be the restriction to $-Y$ of the \spinc structure $\TT_k$ on the neighborhood $N(C)$ characterized by 
\[
 \langle c_1(\TT_k), [C]\rangle - d^2 = 2k
\]
(see \cite{OS:hfk}, where we note that $C$ represents the same homology class as the union of the core of the 2-handle with a Seifert surface for $K$). Since $C$ is rational, we have $\sum g_i = \frac{1}{2}(d-1)(d-2)$, where $g_j$ is the Seifert genus of the knot $K_j$ (see \cite[Equation (2.1)]{BorodzikLivingston} or \cite{milnorsing}). Using additivity of the Seifert genus we find
\[
\langle c_1(\s_\omega), [C]\rangle - d^2 = 2(1-g(K))
\]
Thus we find that $\s_{\xi_C} = \s_{1-g}$ where $g = g(K)$. It also follows that $d^2 > 2g-2$. 

The results of \cite{OS:hfk} say that when $n>2g-2$, the Floer homology $\hfhat(S^3_n(K), \s_k)$ is isomorphic to the homology of a complex $A_k$ constructed from the knot Floer complex $CFK^\infty(K)$. The latter is a doubly-filtered chain complex whose two filtrations we will indicate by $i$ and $j$, and the complex $A_k$ is the subquotient determined by generators for which either $i = 0$ and $j\leq k$, or $i \leq 0$ and $j=k$. Moreover, after a chain homotopy equivalence, we can assume that all generators of $CFK^\infty(K)$ have filtrations satisfying $-g \leq i+j \leq g$. Finally, there is a symmetry $A_k \simeq A_{-k}$ for each $k$, where $\simeq$ is chain homotopy equivalence. 

Let $B = C\{i = 0\}$ be the subquotient of $CFK^{\infty}(K)$ determined by generators with $i = 0$. Recall from \cite{OS:hfk, OS:hfkgenus} that
\begin{itemize}
\item The homology of $B$ is 1-dimensional (we work over $\F$), equal to the Floer homology of $S^3$.
\item The subcomplexes $\eff_k = C\{i=0, j\leq k\}\subset B$ define the {\it knot filtration} on $B$, and the {\it knot Floer homology} of $K$ is the sequence of groups $\hfkhat(K, k) = H_*(\eff_k/\eff_{k-1})$. 
\item If the knot $K$ is fibered then $\dim\hfkhat(K, g) = 1$.
\item The invariant $\tau(K)\in \zee$ is defined to be the smallest integer $k$ so that the map $H_*(\eff_k)\to H_*(B)$ induced by the inclusion of $\eff_k$ is nontrivial.
\item There is a chain isomorphism $U$ of $CFK^\infty(K)$ that maps generators with bifiltration $(i,j)$ to those at bifiltration $(i-1,j-1)$.
\end{itemize}
Each of the knots $K_i$ is an $L$-space knot \cite{Hedden:cabling2}, hence is both fibered and has $\tau(K_i) = g(K_i)$ \cite{BS:Lspaceknots, Hedden:positivity}. Since $\tau$ and Seifert genus are both additive under connected sum, and the fibered property is also preserved under connected sum, we see that $K$ is fibered with $\tau(K) = g$ as well. We claim that any knot with these properties has $H_*(A_{1-g})\cong \F$. Since this homology is isomorphic to $\hfhat(S^3_{d^2}(K), \s_{1-g})$, and the condition that $HF_{red}= 0$ is equivalent to $\hfhat$ being one-dimensional (for a rational homology sphere, in a given \spinc structure), this will complete the proof.

Since $H_*(A_{1-g}) \cong H_*(A_{g-1})$, we work with the latter for convenience. First we claim that under the given hypotheses, $H_*(\eff_{g-1}) =0$. Indeed, we examine the long exact sequence associated to the inclusion $\eff_{g-1}\hookrightarrow B$, observing that $B = \eff_g$:
\[
\cdots \to H_*(\eff_{g-1})\to H_*(B) \to H_*(\eff_g/\eff_{g-1})\to \cdots
\]
Since $\tau(K) = g$, the map $ H_*(\eff_{g-1})\to H_*(B)$ vanishes, while since $K$ is fibered we have $H_*(\eff_g/\eff_{g-1}) = \hfkhat(K, g) \cong \F$. The claim follows since $H_*(B) \cong \F$ as well.

Now we examine the inclusion of $C\{i = -1, j = g-1\}$ as a subcomplex of $A_{g-1}$. The quotient is clearly identified with $\eff_{g-1}$, giving the long exact sequence
\[
\cdots\to H_*(C\{i = -1, j = g-1\}) \to H_*(A_{g-1}) \to H_*(\eff_{g-1})\to \cdots.
\]
The map $U$ induces an isomorphism $\hfkhat(K,g) = H_*(C\{i = 0, j=g\})\cong H_*(C\{i = -1, j = g-1\})$, hence the first term above is one-dimensional. The third term vanishes by the previous claim, hence $H_*(A_{g-1})\cong\F$ as asserted.
\end{proof}

\subsection{Brieskorn spheres in larger rational surfaces}\label{largerationalsec} 

We first observe that any Brieskorn sphere---indeed, any suitably-oriented Seifert manifold over $S^2$---can be found as a contact type hypersurface in some positive symplectic rational surface. To do so, let $Y$ be such a Seifert manifold and choose the orientation so that $Y$ is the boundary of a negative-definite plumbed 4-manifold $X_\Gamma$, where $\Gamma$ is a star-shaped plumbing graph. Write $m = e_0(Y)$ for the framing of the central vertex of $\Gamma$, and let $F_m$ be a rational ruled surface containing a section (embedded symplectic sphere) of self-intersection $m$. Then by repeatedly blowing up sphere fibers of $F_m$, one constructs a configuration $C$ of symplectic spheres in a blowup $\tF_m$, which intersect orthogonally (with respect to the symplectic form) according to the graph $\Gamma$. As a negative-definite symplectic configuration, $C$ has a convex neighborhood $nbd(C) \cong X_\Gamma$, whose boundary is a contact type hypersurface diffeomorphic to $Y$. By choosing the symplectic blowups appropriately, it is easy to ensure that the symplectic structure $\tilde\omega$ on $\tF_m$ is positive. (See \cite[Section 8.1]{SSW} for another perspective on this construction.)

Clearly, the construction above does not yield $X_\Gamma = nbd(C)$ as a Weinstein domain inside $\tF_m$: indeed, since this submanifold contains symplectic spheres, the symplectic form is not exact on $nbd(C)$. On the other hand, as long as $m\leq -2$, the manifold $X_\Gamma$ admits Stein structures by the procedure recalled in Section \ref{nsuitable}, and under some circumstances this structure can be chosen to agree near $\partial X_\Gamma$ with the original symplectic structure (this is essentially a matter of understanding the contact structures induced by the two symplectic forms; we undertake a similar construction below). In this case one can obtain a new symplectic structure $\tilde\omega'$ on $\tF_m$ by replacing the given symplectic structure on $nbc(C)$ with this Stein structure, and in this structure we see that $Y$ bounds a Weinstein domain. Necessarily by Theorem \ref{obstrthm}, the structure $\tilde\omega'$ will no longer be positive.

For a more concrete illustration, we construct a family of examples of symplectic structures on large rational surfaces, which contain (suitable) Brieskorn spheres as boundaries of Weinstein domains. This shows again that ``small'' is a necessary condition in Corollary \ref{nosuitablesmall} (and ``positive'' is necessary in Corollary \ref{nosuitablepositive}).  Various other examples along these lines are easy to obtain; we do not attempt to be exhaustive.

Consider a rational unicuspidal curve, i.e. one that has just one cusp singularity, with link the knot $K\subset S^3$. We also suppose that the singularity has a single Puiseux pair, which topologically corresponds to the case that $K$ is a torus knot. If the curve has degree $d$, then a neighborhood of $C$ is diffeomorphic to the result of attaching a handle to $B^4$ along $K$ with framing $d^2$. Choose $d^2 + 1$ smooth points in $C$, and blow up $\cee P^2$ at each of these points. We can perform the blowups symplectically, so as to obtain a symplectic (or even K\"ahler) form $\tomega$ on the blowup $\cee P^2\# (d^2 + 1)\cptwobar$. The proper transform $\tC$ of $C$ has self-intersection $\tC\cdot\tC = -1$, and its neighborhood $N(\tC)$ inside $\cee P^2\# (d^2 + 1)\cptwobar$ is diffeomorphic to a 4-ball with a handle attached along $K$ with framing $-1$. Now, from our assumptions, $K$ is isotopic to a positive torus knot $T_{p,q}$ for some $p,q$. Hence $\partial N(\tC)$ is diffeomorphic to the surgery manifold $S^3_{-1}(T_{p,q})$, which is the Brieskorn homology sphere $\Sigma(p,q, pq+1)$. 

All  positive torus knots have nonnegative maximal Thurston-Bennequin invariant, so in particular $K$ admits a Legendrian representative $\K$ with Thurston-Bennequin number $\tb(\K) = 0$, obtained for example by suitably stabilizing a standard representative. Thus the neighborhood $N(\tC)$ is diffeomorphic to the Stein manifold obtained by attaching a Stein 2-handle along $\K$, and in particular we see $N(\tC)$ admits a Stein structure. Of course, the symplectic structure $\omega_{\st}$ associated with a Stein structure on $N(\tC)$ is not equivalent to the restriction of the ambient symplectic structure $\omega_{\symp} = \tomega|_{N(\tC)}$: indeed, the curve $\tC$ is symplectic for $\omega_{\symp}$, and in particular the cohomology class $[\omega_{\symp}]$ is nontrivial, while $\omega_{\st}$ is exact by construction. However, we have the following.

\begin{lemma} The neighborhood $N(\tC)$ can be chosen to be symplectically convex for both $\omega_{\symp}$ and $\omega_{\st}$. For suitable choice of the stabilizations used to obtain the Legendrian $\K$, the corresponding contact structures $\xi_{\symp}$ and $\xi_{\st}$ on $\Sigma(p,q, pq+1) = \partial N(\tC)$ are isotopic. Hence the forms $\omega_{\symp}$ and $\omega_{\st}$ can be chosen to agree on a neighborhood of $\partial N(\tC)$.
\end{lemma}

\begin{proof}
The claim that $N(\tC)$ can be chosen to be symplectically convex for $\omega_{\symp}$ follows from \cite[Theorem 1.2]{GayStipsicz2009}, by resolving the singular point of $\tC$ by blowups. For the remaining claims we first briefly recall the classification of fillable contact structures on the Brieskorn homology sphere $\Sigma(p, q, pq+1)$, which we have observed is the result of smooth $-1$ surgery on the $(p,q)$-torus knot $T_{p,q}$. The knot $T_{p,q}$ admits an oriented Legendrian representative $\mathcal K'$ with maximal Thurston-Bennequin number $tb(\mathcal K')=pq-p-q$ and $rot(\mathcal K')=0$. In particular, any contact $(-pq+p+q-1)$ surgery along $\mathcal K'$ (smoothly, $-1$ surgery) produces a Stein fillable contact structure on the smooth manifold $\Sigma(p,q, pq+1)$, each of which is filled by a Stein structure on $N(\tC)$ obtained by a Stein 2-handle attachment. More precisely, such a contact surgery is realized as Legendrian surgery on a Legendrian knot $\mathcal K$ that is obtained from $\mathcal K'$ by $pq-p-q$ stabilizations, each of which can be chosen to be either positive or negative.  Thus there are $pq-p-q+1$ such Legendrians $\K$, having pairwise distinct rotation numbers $\{-pq+p+q, -pq+p+q+2, \cdots, pq-p-q-2, pq-p-q\}$ . Since the Chern classes of the corresponding Stein structures evaluate on the handle as the rotation numbers of Legendrian knots, by a result due to Lisca-Mati\'{c} \cite[Theorem 1.2]{LM}, the induced contact structures are non-isotopic. This produces at least $pq-p-q+1$ fillable structures on $\Sigma(p,q,pq+1)$. Finally, by \cite[Theorem 1.9]{MT:pseudoconvex} for $(p,q)=(2,3)$ and \cite{EtnyreMinTosunVarvarezosPre} for general $(p,q)$, we know that $\Sigma(p,q,pq+1)$ has at most (and hence exactly) $pq-p-q+1$ fillable contact structures. Since $\xi_{\symp}$ is such a structure, we can choose the Stein structure on $N(\tC)$ so that $\xi_{\symp}$ and $\xi_{\st}$ on $\Sigma(p,q, pq+1) = \partial N(\tC)$ are isotopic.
\end{proof}

As with the construction at the beginning of this section, the convexity of $N(\tC)$ with respect to $\tilde\omega$ exhibits $\Sigma(p,q,pq+1)$ as a contact type hypersurface in $(\cee P^2\# (d^2+1)\cptwobar, \tilde\omega)$, which by appropriate choices in the blowup construction we can take to be positive.

We now define a symplectic form $\omega_0$ on $X := \cee P^2\# (d^2 + 1)\cptwobar$ by setting $\omega_0 = \omega_{\symp}$ outside $N(\tC)$, and $\omega_0 = \omega_{\st}$ in $N(\tC)$. This is a well-defined smooth symplectic form by the lemma, and by construction $\Sigma(p,q,pq+1)$ is a contact type hypersurface in $(X,\omega_0)$ whose convex side is the Stein domain $(N(\tC), \omega_{\st})$. In particular $(N(\tC),\omega_{\st})$ is a Weinstein domain symplectically embedded in $(X,\omega_0)$ with contact boundary $(\Sigma(p,q,pq+1), \xi_{\st})$. 

We have seen in Corollary \ref{suitablecor} that $\Sigma(p,q,pq+1)$ is suitable. 
It follows by Corollary \ref{nosuitablepositive} that the symplectic structure $\omega_0$ on $X = \cee P^2\# (d^2 + 1)\cptwobar$ is non-positive, though a more direct argument is possible. (Note that by Theorem \ref{obstrthm}, to see the non-positivity of $\omega_0$ it would suffice to show only that the particular contact structure $\xi_{\st}$ on $\Sigma(p,q,pq+1)$ has $c_{red}(\xi_{\st}) \neq 0$. )

\begin{remark} There is a full topological classification of rational unicuspidal plane curves with a single Puiseux pair, due to Fern\'andez de Bobadilla, Luengo, Melle-Hern\'andez, and N\'emethi \cite{FLMN}. The simplest example among the possibilities is a curve of degree 3, whose singularity is modeled on the cone on the right trefoil $T(2,3)$. This gives rise by the procedure above to a symplectic structure $\omega_0$ on $\cee P^2 \# 10\cptwobar$ containing a Weinstein domain whose boundary is $\Sigma(2,3,7)$. More generally, the family of \cite[Theorem 1.1(a)]{FLMN} gives rise by this construction to a copy of $\Sigma(d-1, d, d^2-d+1)$ that bounds a Weinstein domain in an appropriate, non-positive, symplectic structure on $\cee P^2 \# (d^2 + 1)\cptwobar$, for each $d\geq 3$.
\end{remark}

\subsection{Additional examples} \label{planarsec}

In \cite{Oba:mazur}, Oba constructed an infinite family of Stein manifolds $W_n$ with the following properties:
\begin{itemize}
\item Each $W_n$ is the total space of a positive allowable Lefschetz fibration over $D^2$ whose typical fiber is a disk with 3 holes, and which has three vanishing cycles.
\item The $W_n$ are contractible, and are Mazur type in the sense that they admit a handle decomposition with a single handle of each index $0,1,2$.
\item The boundary 3-manifolds $Y_n = \partial W_n$ are mutually non-diffeomorphic.
\end{itemize}

We can embed each $W_n$ symplectically in a rational surface by following the procedure of \cite{Etnyre:planar}: one attaches symplectic 2-handles along circles in $Y_n$ corresponding to all of the boundary components of the Lefschetz fiber, to arrive at a symplectic 4-manifold containing $W_n$, with a genus-0 Lefschetz fibration over $D^2$ having closed fibers, and whose boundary is $S^1\times S^2$. This can be completed to a closed symplectic manifold $X$ by adding the neighborhood of a symplectic sphere having self-intersection $0$. (Alternatively, one can use the construction of \cite{Wendl:nonexact} to cap off all but one boundary component of the fiber, resulting in a symplectic manifold strongly filling $(S^3,\xi_{std})$, and complete with a neighborhood of a symplectic $+1$ sphere.) By the classification of symplectic 4-manifolds containing symplectic spheres of nonnegative square \cite{McDuff:rationalruled}, and since $X$ is simply connected, $X$ is a rational surface. In fact, since there are three vanishing cycles in each $W_n$, a simple calculation in homology shows that $X\cong \cee P^2 \# 3 \cptwobar$. We conclude:

\begin{proposition}\label{planarprop} For each $n\geq 1$, there is a symplectic structure on $X = \cee P^2 \# 3\cptwobar$ containing a symplectically embedded copy of the Stein manifold $W_{n}$, and in particular $X$ contains the 3-manifold $Y_{n}$ as a contact type hypersurface bounding a Weinstein domain.
\end{proposition}

As the boundary of a nontrivial Mazur manifold, $Y_n$ never has vanishing reduced Floer homology in its unique \spinc structure \cite{ConwayTosun}.

Clearly the same construction embeds any genus-0 positive allowable Lefschetz fibration as a Weinstein submanifold of a rational surface. It is not hard to see that ``interesting'' examples (say, those having hyperbolic 3-manifolds as boundary) embed by this construction in $\cee P^2 \# k\cptwobar$ with $k\geq 3$.

\section{Proof of Theorem \ref{ratconvthm}}\label{proofsec2}

\subsection{Background and main argument}
In this section we prove that Weinstein domains in symplectic manifolds are ``symplectically rationally convex'' in the sense of Theorem \ref{ratconvthm}. As mentioned in the introduction our argument closely parallels that of Auroux, Gayet and Mohsen \cite{AGM01}, which itself is based on techniques of Donaldson \cite{don96}. The argument applies in all dimensions, so we suppose $(X,\omega)$ is a closed symplectic manifold of dimension $2n$ and $W\subset X$ a Weinstein domain. This means $W$ is part of a tuple $(W, \omega, \varphi, V)$, where we do not distinguish between $\omega$ and its restriction to $W$. Here $V$ is a Liouville vector field, meaning that $\L_V\omega = \omega$, and hence the 1-form $\lambda = \iota_V\omega$ satisfies $d\lambda = \omega$. We also choose an almost-complex structure $J$ and Riemannian metric $g$ on $X$ that are compatible with $\omega$.

The condition that $\frac{1}{2\pi}[\omega] \in H^2(X;\arr)$ lies in the image of the integral cohomology of $X$ means that we can find a complex line bundle $L\to X$ whose real first Chern class is $c_1(L) = \frac{1}{2\pi}[\omega]$.  Fixing such $L$, we also choose a hermitian structure on $L$ and a unitary connection $\nabla$ whose curvature form is $F = -i\omega$. For each $k>0$, we consider the bundle $L^{\otimes k}$ with induced metric and connection, whose curvature is $F_k = -ik\omega$. The almost-complex structure and connection determine operators $\partial$ and $\dbar$ on $L^\tk$ in the usual way, and if $s$ is a section of $L^\tk$, we can consider the norms $|s|$, $|\dbar s|$, $|\nabla s|$, $|\nabla\nabla s|$ and so on, where the norms of derivatives involve the metric $g$. As in \cite{AGM01} and \cite{don96}, in some circumstances it is convenient to use the scaled metric $g_k = k\cdot g$ instead, and in that case we write $|\cdot |_k$ for the corresponding norm.

Suppose that for each $k$ we are given a section $\sigma_k$ of $L^\tk$. Following \cite{AGM01}, we say that the sequence $\{\sigma_k\}$ is {\it asymptotically holomorphic} if there is a constant $C>0$ independent of $k$ such that
\[
|s_k| + |\nabla s_k|_k + |\nabla\nabla s_k|_k<C\quad\mbox{and}\quad |\dbar s_k|_k + |\nabla \dbar s_k|_k < Ck^{-1/2}
\]
pointwise on $X$, for each $k$. Donaldson observed that if the sequence $\{\sigma_k\}$ is asymptotically holomorphic and satisfies an additional condition called {\it uniform transversality to 0}, then for all sufficiently large $k$ the zero locus $\Sigma_k = \sigma_k^{-1}(0)$ is an embedded symplectic submanifold of $X$. By construction, the homology class $[\Sigma_k]$ is Poincar\'e dual to $\frac{k}{2\pi}[\omega]$. Auroux, Gayet and Mohsen showed that an isotropic submanifold $L$ of $X$ is symplectically rationally convex by proving that for each $p\in X-L$ one can find a sequence $\{\sigma_k\}$ that is asymptotically holomorphic and uniformly transverse to 0, where each $\sigma_k(p)= 0$ and furthermore each $\sigma_k$ is nonvanishing on $L$. Our goal is similar, with the Weinstein domain $W$ replacing the isotropic submanifold $L$ (and where we will have to allow the contraction of $W$). 

The construction of the needed sections goes roughly as follows. Donaldson first observed that the positivity of the curvature of $L^\tk$ allows one to construct sections that are controllably ``close'' to being holomorphic and also supported near any given point $p\in X$. To obtain such a section one takes a Darboux chart $\psi: B\subset \cee^n\to X$ centered at $p$, whose domain is a ball $B$ of fixed size. One can arrange that $\psi$ is $J$-holomorphic at $p$, and that $|\dbar_J \psi|$ is bounded by a constant independent of $p$, using compactness of $X$. One can take $L^\tk$ to be trivialized over $\psi(B)$, in such a way that the connection 1-form becomes $A_k = \frac{k}{4}\sum_j (z_j\,d\bar z_j - \bar z_j \,dz_j)$ in the chart coordinates. Then the local section defined in $B$ by
\[
\hat\sigma_k(z) =e^{-k|z|^2/4}
\]
is holomorphic, i.e. satisfies $\dbar_{A_k}\hat\sigma_k = 0$ where $\dbar_{A_k}$ is defined using the standard complex structure on $\cee^n$. Introduce a cutoff function $\beta_k$ that vanishes at distance $k^{-1/3}$ from $0$ in $\cee^n$, and transfer the function $\beta_k\hat\sigma_k$ to $X$ via $\psi$ to obtain a global section $\sigma_{p,k}$ of $L^\tk$ that vanishes outside a neighborhood of $p$ having $g$-diameter $O(k^{-1/3})$, or equivalently $g_k$-diameter $O(k^{1/6})$. By a suitable scaling, we can assume that $|\sigma_{p,k}(p)| = 1$. One can show that as $k$ increases the family $\{\sigma_{p,k}\}$ is asymptotically holomorphic and is ``concentrated'' at $p$ in the sense of \cite{AGM01} (see \cite[Section 2]{don96}). 

To continue, one chooses for each $k$ a finite subset $P_k\subset X$, such that the $g_k$-unit balls centered at points of $P_k$ cover $X$, but whose points are not ``too clustered''---e.g., no pair of points are within $g_k$ distance $\frac{2}{3}$. Then for any collection $w = \{w_p\}_{p\in P_k}$ of complex numbers, one obtains a section
\begin{equation}\label{sectiondef}
s_w = \sum_{p\in P_k} w_p\,\sigma_{p,k}
\end{equation}
of $L^\tk$ which, as long as each $|w_p|\leq 1$, Donaldson shows gives rise to an asymptotically holomorphic family. Donaldson further shows that by a suitable perturbation procedure, given an asymptotically holomorphic sequence $\{s_k\}$, one can make a $C^1$-small (with respect to $g_k$) modification of the sections $s_k$ so that the perturbed sequence $\{\tilde s_k\}$ is asymptotically holomorphic and uniformly transverse to 0 (see also \cite{aurouxGAFA97} and the discussion before Lemma 2 in \cite{AGM01}). Our primary goal, then, is to see that we can find an asymptotically holomorphic sequence $\{s_k\}$ so that $|s_k|>c$ at all points of $W_t$, for all sufficiently large $t$ and some constant $c$ independent of $k$. A sufficiently small perturbation then gives an asymptotically holomorphic and uniformly transverse family $\{\tilde s_k\}$ such that $\tilde s_k$ does not vanish on $W_t$ for all sufficiently large $t$. It follows that for sufficiently large $k$, the subset $\Sigma = s_k^{-1}(0)$ is a symplectic submanifold Poincar\'e dual to $\frac{k}{2\pi}[\omega]$, which is disjoint from $W_t$. This will prove Theorem \ref{ratconvthm}, except for requiring $\Sigma$ to contain a given point $p$.

The essential point is to show that the coefficients of the section $s_w$ can be chosen so that any cancellation that occurs in the combination \eqref{sectiondef} is negligible over $W_t$ for large $t$. The following lemma is the key to making this argument; the statement is an adaptation of \cite[Lemma 2]{AGM01}.

\begin{lemma}\label{trivlemma} There exists a sequence of positive real numbers $t_k\to\infty$ and a constant $C>0$ such that whenever $k$ is a multiple of $|\mathrm{Tor}(H_1(W;\zee))|$, the restriction of $L^{\otimes k}$ to $W_{t_k}$ admits a smooth unit section $\tau_k$ having $|\nabla \tau_k|\leq C$ at each point of $W_{t_k}$.
\end{lemma}

Note that if $k$ satisfies the condition in the statement then $L^\tk$ is topologically trivial over $W$. We assume $k$ satisfies this condition from now on.

Granted Lemma \ref{trivlemma}, one uses the sections $\tau_k$ to fix the arguments of the coefficients $w_p$ appearing in \eqref{sectiondef}. More explicitly, and following \cite{AGM01}, for each $k$ we let $P_k$ be a finite subset of the form $P_k = P_k' \cup P_k ''$, where $P_k'$ is a subset of $W_{t_k}$ such that the $g_k$-unit balls of radius 1 cover $W_{t_k}$, and any two points of $P_k'$ are $g_k$-distance at least $\frac{2}{3}$ apart, and $P_k''$ lies in $X - W_{t_k}$ (we suppose the whole collection $P_k'\cup P_k''$ satisfies the covering and non-clustered conditions over all of $X$). 
Then we define sections
\[
s_{k,W} = \sum_{p\in P_k'} \frac{\tau_k(p)}{\sigma_{k,p}(p)} \sigma_{k,p}
\]
of $L^{\otimes k}$ over $X$, where $\tau_k$ is the trivialization of $L^{\otimes k}$ over $W_{t_k}$ of Lemma \ref{trivlemma}. (Thus, the coefficients $w_p = 0$ when $p\in P_k''$.)  The proofs of \cite[Lemmas 3 and 4]{AGM01} carry over verbatim to this situation to show that the section $s_{k,W}$ is ``concentrated'' over $W_{t_k}$ in the sense of \cite{AGM01}, and in particular $|s_{k,W}|$ satisfies a lower bound over $W_{t_k}$ independent of $k$. Then replacing $s_{k,W}$ by suitable $C^1$-small perturbations as above, we obtain the desired family $\{\tilde s_{k,W}\}$.

To see that we can arrange $\tilde s_{k,W}(p_0) = 0$ for any chosen $p_0\in X - W$, we argue just as in \cite{AGM01}. Namely, instead of the section $\sigma_{k,p_0}$ supported near $p_0$ we use $u_{k,p_0} = k^{1/2} z_1 \sigma_{k, p_0}$ where $(z_1,z_2)$ are local approximately holomorphic coordinates centered at $p_0$. Then replace $\{s_{k,W}\}$ by the asymptotically holomorphic sequence $\{s_{k,W} + u_{k, p_0}\}$. For large enough $k$, the support of $u_{k,p_0}$ is disjoint from $W_{t_k}$, and therefore taking small perturbations of the new sequence yields symplectic hypersurfaces $\Sigma_k$ that are disjoint from $W_{t_k}$ and pass through a point $x_k$ that is within $g_k$-distance 1 of $p_0$. We can ensure $\Sigma_k$ contains $p_0$ itself by applying a Hamiltonian diffeomorphism supported away from $W_{t_k}$ and moving $x_k$ to $p_0$.

This completes the proof of Theorem \ref{ratconvthm}, given Lemma \ref{trivlemma}.

\subsection{Proof of Lemma \ref{trivlemma}}

Recall that if $(W, \omega, \varphi, V)$ is a Weinstein structure on $W$, then the {\it skeleton} of $W$ is the closed subset
\[
\Skel(W) = \bigcap_{t>0} W_t
\]
where we recall our convention that $W_t = \Phi_{-t}(W)$ for $t>0$, and $\Phi_t$ is the time-$t$ flow of the Liouville field $V$.  By a homotopy of Weinstein structures we can assume that the flow of $V$ is Morse-Smale, and then by adjusting $\varphi$ suitably, we can suppose the critical values of $\varphi$ occur in increasing order of Morse index (see for example \cite[Proposition 12.12, Lemma 12.20]{CE}). Under these conditions, we have that the skeleton of $W$ has the structure of a $CW$ complex $C\subset W$, whose open cells are the descending manifolds of the critical points of $\varphi$ \cite[Section 4.9]{AudinDamian} and which are smooth isotropic submanifolds \cite[Lemma 11.13]{CE}.

\begin{lemma}\label{liouvillelemma} Let $\epsilon >0$, and $U$ an open neighborhood of $C$. There exists a 1-form $\lambda_\epsilon$ on $W$ such that 
\begin{enumerate}
\item $\lambda_\epsilon$ is a Liouville form, i.e. $d\lambda_\epsilon = \omega$ on $W$.
\item $\lambda_\epsilon = \lambda$ outside $U$, where $\lambda = \iota_V \omega$.
\item We have $|\lambda_\epsilon| < \epsilon$ at each point of  $C$.
\end{enumerate}
In particular there is a neighborhood $U_\epsilon\subset U$ of $C$ such that $\lambda_\epsilon = \lambda$ outside $U$ and $|\lambda_\epsilon|< \epsilon$ on $U_\epsilon$.
\end{lemma}

Note that if the skeleton $C$ is a smooth submanifold, then one can find a Liouville form as in the statement that actually vanishes everywhere along $C$. In general, however, this need not be possible.

\begin{proof}
We begin by recalling an elementary portion of a more detailed construction from \cite[Section 12.3]{CE}. Let $\lambda$ and $\alpha$ be Liouville forms for $\omega$, and $f$ a smooth function on $W$. We suppose the support of $f$ is contained in some open set $U$, and require only that $\alpha$ is defined on $U$. Finally, assume $H$ is a function  with $dH = \alpha - \lambda$ on $U$. Then the 1-form $\tlambda$ given by
\[
\tlambda = (1-f)\lambda + f\alpha + H\, df
\]
is easily seen to define a smooth Liouville form on $W$ such that $\tlambda = \lambda$ on any open set on which $f = 0$, and $\tlambda = \alpha$ on any open set on which $f = 1$. In fact, we have $\tlambda = \alpha$ whenever $f = 1$ and $H = 0$.

Now suppose $K\subset L$ is a compact subset of an isotropic submanifold $L\subset W$, and suppose for convenience that the symplectic normal bundle $TL^\omega/ TL$ is trival of rank $2\ell$. We also assume that the given Liouville form $\lambda$ on $W$ has $\lambda|_L = 0$, or equivalently that the corresponding Liouville vector field is tangent to $L$ (both of these assumptions hold in our application below). Then by standard neighborhood theorems, $K$ has a neighborhood  in $W$ symplectomorphic to a neighborhood of $K$ in $T^*L \oplus \arr^{2\ell}$, equipped with a standard symplectic form. In particular, taking coordinates $\{q_i\}$ on $L$ and corresponding coordinates $\{p_i\}$ in the fibers of $T^*L$, along with symplectic coordinates $(x_j,y_j)_j$ on $\arr^{2\ell}$, the form $\alpha = -\sum_i p_i\, dq_i+ \frac{1}{2} \sum_{j = 1}^\ell (x_j \, dy_j - y_j \, dx_j)$ defines a Liouville form in a neighborhood of $K$, which vanishes along $K$. Since $\lambda|_L = 0$, the relative Poincar\'e lemma implies that there is a function $H$ on a neighborhood $U_1$ of $K$ so that $dH = \alpha -\lambda$ and also $H = 0$ along $L\cap U_1$ (see, for example, \cite[Proposition 6.8]{CdS}). Then taking $f$ to be a bump function equal to $1$ near $K$ and vanishing outside $U_1$, the construction above gives a Liouville form $\tlambda$ such that:
\begin{enumerate}
\item[(a)] $\tlambda = \lambda$ outside $U_1$
\item[(b)] $\tlambda = \alpha$ on a neighborhood of $K$, in particular $\tlambda_x = 0$ for $x \in K$.
\item[(c)] $|\tlambda|\leq |\lambda|$ on $L\cap U_1$
\end{enumerate}

We now construct the desired form $\lambda_\epsilon$ by inductive modification near the cells of $C$. Near a 0-cell $e_0$, it is simple to apply the above construction in a Darboux chart, taking $K = L = e_0$ and $\alpha$ a standard Liouville form vanishing at the origin. (Alternatively, since 0-cells are minima of $\phi$ and hence by the gradient-like condition are points at which $\lambda = 0$, the given Liouville form already satisfies conditions (1) and (2) of the statement, and condition (3) on the 0-skeleton.)  Now assume that $C = C_0 \cup e_k$, where $e_k$ is a $k$-cell attached to a subcomplex $C_0\subset C$. Suppose a Liouville form $\lambda_0$ has been found such that $|\lambda_0|< \epsilon$ on $C_0$, and $\lambda_0 = \lambda$ outside the neighborhood $U$ of $C$. By continuity we can find an open neighborhood $U_0\subset U$ of $C_0$ on which $|\lambda_0|< \epsilon$. 

Let $\chi: D^k\to C\subset W$ be a characteristic map for $e_k$. Writing $L = \chi(\Int(D^k))$, we have by the remarks previously that $L$ is a smooth isotropic submanifold of $W$, and that $\lambda|_L = 0$.

Let  $A = \chi^{-1}(U_0) \subset D^k$. Since $\chi(\partial D^k)\subset C_0\subset U_0$, we see that $A$ is an open set containing $\partial D^k$. Choose a compact subset $T \subset \Int(D^k)$ so that $A\cup \Int(T) = D^k$, and let $K = \chi(T)\subset L$. Now apply the construction from above, using the compact set $K\subset L$ and a neighborhood $U_1$ of $K$ in $W$ that we may assume is contained in the given neighborhood $U$ of $C$. We obtain a Liouville form $\tlambda$ on $W$ satisfying the properties (a)-(c) above, which clearly also has properties (1) and (2) of the statement. By construction the sets $C_0\cap U_0$ and $K$ cover $C$; on $C_0\cap U_0$ we have that $|\tlambda|\leq |\lambda|< \epsilon$, while $\tlambda$ vanishes on $K$. This verifies (3) of the statement, setting $\lambda_\epsilon = \tlambda$. 
\end{proof}

\begin{proof}[Proof of Lemma \ref{trivlemma}]
By repeated application of Lemma \ref{liouvillelemma}, for each $k$ we can find a Liouville form $\lambda_k$ so that $|\lambda_k|< \frac{1}{k}$ on the skeleton $C$, and in fact a sequence of neighborhoods $U_1\supset U_2\supset \cdots$ of $C$ so that $|\lambda_k|< \frac{1}{k}$ on $U_k$. 

For each $k$ we choose an arbitrary unitary trivializing section $\sigma_k$ of $L^{\otimes k}$ (over $W$), and let $\alpha_k$ be the corresponding connection 1-forms determined by $\nabla\sigma_k = \alpha_k\sigma_k$. Observe that
\[
d \alpha_k = -ik\omega = -ikd\lambda_k,
\]
so that the form $\alpha_k + ik\lambda_k\in \Omega^1(W; i\arr)$ is closed for each $k$. Note that while the Liouville forms $\lambda_k$ vary with $k$, our deformation of $W$ will be carried out using the original Liouville vector field $V$.

Observe that replacing $\sigma_k$ by $\rho\sigma_k$ for a smooth function $\rho: W\to S^1$ replaces $\alpha_k$ by $\alpha_k + \rho^*{d\theta}$. In particular by choosing $\rho$ of the form $e^{if}$, we can replace $\alpha_k + ik\lambda_k$ by any 1-form in its de Rham class. It is well known that on a compact Riemannian manifold $W$ with boundary, the de Rham cohomology is naturally isomorphic to the (finite-dimensional) space $\H^*_N(W)$ of harmonic forms with Neumann boundary conditions (see for example \cite[Theorem 2.6.1]{schwarzbook}). Hence, adjusting $\sigma_k$ by a topologically trivial gauge transformation $e^{if}$, we can arrange that $\alpha_k + ik\lambda_k$ lies in $i\H^1_N(W)$. Moreover, the remaining gauge action (by topologically nontrivial $\rho$) allows translation of $\alpha_k + ik\lambda_k$ by any element of the image of $H^1(W;i\zee)$ in $i\H^1_N(W)$. Thus by suitable choice of the trivializations $\sigma_k$ we can assume that the forms $\alpha_k +ik\lambda_k$ lie in a subset of $i\H^1_N(W)$ that is bounded independent of $k$. With this condition, we can find a constant $C' >0$ such that there is a pointwise bound $|\alpha_k + ik\lambda_k|_g \leq C'$  over $W$ for all $k$. 

Next observe that the flow $\Phi_{-t}$ carries $W$ inside any given neighborhood of the skeleton $C$, for all sufficiently large $t$. In particular we can choose an increasing sequence $\{t_k\}$ so that $\Phi_{-t_k}(W)\subset U_k$ for each $k$. Thus we see that both $|\alpha_k + ik\lambda_k|$ and $|k\lambda_k|$ are bounded over $W_{t_k}$ by a constant independent of $k$ (in fact $|k\lambda_k|\leq 1$ on $W_{t_k}$), and therefore the same is true for $|\alpha_k|$. Since $|\nabla \sigma_k| = |\alpha_k|$, the lemma follows.
\end{proof}

\begin{remark} From the proof in this section, we see that the conclusion of Theorem \ref{ratconvthm} holds if one assumes only that $W\subset X$ is a Liouville domain admitting a Liouville field $V$ for which the associated skeleton is a smooth isotropic $CW$ complex. We do not know if such a Liouville domain is necessarily Weinstein.
\end{remark}

\subsection{Non-integral symplectic classes}\label{nonintsec}
We now show how to remove the condition that $\frac{1}{2\pi}[\omega]$ be an integral class, by an argument essentially the same as that in \cite[Section 3.2]{AGM01}. More precisely, we have:

\begin{lemma}\label{perturblemma} Let $(X,\omega)$ be a symplectic $2n$-manifold and $W\subset X$ a Weinstein domain symplectically embedded in $X$. Then there exists a symplectic form $\omega'$ on $X$ with the following properties:
\begin{enumerate}
\item The class $\frac{1}{2\pi}[\omega']$ is integral.
\item $W$ is Weinstein in $(X,\omega')$, in fact $\omega'|_W = \omega|_W$.
\item $\omega'$ is a positive multiple of a form arbitrarily $C^0$-close to $\omega$.
\item $c_1(X,\omega') = c_1(X,\omega)$.
\end{enumerate}
\end{lemma}
Thus, if $(Y,\xi)$ is the contact boundary of a Weinstein domain symplectically embedded in a positive symplectic 4-manifold $(X^4,\omega)$, then $(Y,\xi)$ also bounds a Weinstein domain in $(X,\omega')$ where $\frac{1}{2\pi}[\omega']$ is integral (and $\omega'$ is still positive).

\begin{proof}
Since $\omega$ is exact on $W$, we see that $[\omega]$ lies in the image of the homomorphism $j^*: H^2(X,W)\to H^2(X)$ in de Rham cohomology, where $j$ is the inclusion of $X$ in the pair $(X,W)$. Writing $[\omega] = j^*[\alpha]$ for a 2-form $\alpha$ vanishing on $W$, observe that we can find an arbitrarily small closed 2-form $\alpha_0$ also vanishing on $W$ so that $\frac{1}{2\pi}[\alpha + \alpha_0]$ lies in the image of $H^2(X, W;\cue) \to H^2(X,W;\arr)$. Furthermore, for any sufficiently small closed $\alpha_0$, the form $\omega + \alpha_0$ is symplectic on $X$. Thus the class $\frac{1}{2\pi}[\omega+ \alpha_0] = \frac{1}{2\pi}j^*[\alpha + \alpha_0]$ lies in $H^2(X;\cue)$, and setting $\omega' = N(\omega + \alpha_0)$ for suitable positive integer $N$ gives the desired symplectic form.

If $J$ is an almost-complex structure compatible with $\omega$, then  $\omega'$ admits a compatible almost-complex structure $J'$ that is $C^0$-close to $J$. Since the space of almost-complex structures is locally connected, (4) follows.
\end{proof}

This argument also shows that a Weinstein domain $W\subset X$ always has symplectic hypersurfaces passing through any given point in its complement after a suitable contraction. Indeed, we replace $\omega$ by $\omega'$ and $J$ by $J'$ as above, and (following \cite{AGM01}) observe that since $\omega(v, J'v)>0$, any $J'$-holomorphic subspace of $TX$ is $\omega$-symplectic. Hence a sequence $\Sigma_k$ of $J'$-almost-holomorphic submanifolds is $\omega$-symplectic for large enough $k$. Since $\omega = \omega'$ on $W$ we can suppose the Weinstein structure on $W$ is unchanged, so the contraction of $W$ is not affected by the modification. Furthermore, since the Chern class of $\omega'$ is the same as that of $\omega$, the argument using the adjunction inequality in the proof of Theorem \ref{obstrthm}(b) still applies to show $\Sigma_k\cdot \Sigma_k > \max\{0, 2g(\Sigma_k) - 2\}$.

\bibliography{references}
\bibliographystyle{amsplain}

\end{document}